\documentclass[english]{article}
\usepackage[T1]{fontenc}
\usepackage[latin9]{inputenc}
\usepackage{amsthm}
\usepackage{amsmath,mathtools}
\usepackage{amssymb}
\usepackage{mathrsfs}
\usepackage{esint}
\usepackage{babel}
\usepackage{vmargin}

\usepackage{ifpdf}
\ifpdf
\usepackage[pdftex]{hyperref}
\else
\usepackage[hypertex]{hyperref}
\fi

\hypersetup{
citecolor=green,
menucolor=red,
citebordercolor=0 1 0,
linkbordercolor=1 0 0,
pdfpagemode=UseOutlines,
pdftitle={Multi-solitary waves for the nonlinear Klein-Gordon equation},
pdfauthor={ Jacopo Bellazzini <jbellazzini@uniss.it>, Marco Ghimenti<ghimenti@mail.dm.unipi.it>,
Stefan Le Coz <slecoz@math.univ-toulouse.fr>},
pdfsubject={Multi solitary waves for the nonlinear Klein-Gordon equation},
pdfkeywords={solitary waves, Klein-Gordon equation} 
}

\newtheorem{proposition}{Proposition}
\newtheorem{lemma}[proposition]{Lemma}
\newtheorem{theorem}{Theorem}

\newtheorem*{lwp1}{Cauchy Theory in $H^1\times L^2$}
\newtheorem*{lwp2}{Cauchy Theory in $H^s\times H^{s-1}$}
\newtheorem*{FPS}{Finite Propagation Speed}
\newtheorem*{BA}{Bootstrap Assumption}

\newtheorem*{keywords}{Keywords}
\newtheorem*{MSC2010}{MSC2010}
\newtheorem*{ack}{Acknowledgments}

\theoremstyle{remark}
\newtheorem{remark}{Remark}

\DeclarePairedDelimiter{\norm}{\lVert}{\rVert}

\DeclareMathOperator{\Span}{Span}

\newcommand{\dual}[2]{\left\langle #1,#2\right\rangle}
\newcommand{ \psld}[2]{\left(#1,#2\right)_2}

\newcommand{\R}{\mathbb R}

\renewcommand{\Re}{\operatorname{Re}}
 \renewcommand{\Im}{\operatorname{Im}}

\begin{document}

\title{Multi-solitary waves \\for the nonlinear Klein-Gordon equation}

\author{ Jacopo Bellazzini \thanks{ Universit\`a degli Studi di Sassari, Via
Piandanna 4, Sassari, ITALY. e-mail: \texttt{jbellazzini@uniss.it}}   \and  Marco Ghimenti \thanks{
Dipartimento di Matematica, Universit\`a di Pisa, Largo Bruno Pontecorvo 5, Pisa, ITALY. e-mail: \texttt{ghimenti@mail.dm.unipi.it}} \and  Stefan Le Coz \thanks{Institut de Math\'ematiques de Toulouse, Universit\'e Paul Sabatier 118 route de Narbonne, 31062 Toulouse Cedex 9 FRANCE. e-mail: \texttt{slecoz@math.univ-toulouse.fr}}}

\maketitle

\begin{abstract}
\noindent We consider the nonlinear Klein-Gordon equation in $\R^d$. We call multi-solitary waves a solution behaving at large time as a sum of boosted standing waves. Our main result is the existence of such multi-solitary waves, provided the composing boosted standing waves are stable. It is obtained by solving the equation backward in time around a sequence of approximate multi-solitary waves and showing convergence to a solution with the desired property. The main ingredients of the proof are finite speed of propagation, variational characterizations of the profiles, modulation theory and energy estimates.
\end{abstract}

\begin{keywords}
Klein-Gordon equation, multi-soliton, asymptotic behavior
\end{keywords}

\begin{MSC2010}
35Q51, 35C08, 35Q40, 35L71, 37K40.
\end{MSC2010}

\section{Introduction}


 We consider the following nonlinear Klein-Gordon equation 
\begin{equation}\label{eq:nlkg}\tag{\textsc{nlkg}}
u_{tt}-\Delta u+mu-|u|^{p-1}u=0
\end{equation}
where $u:\mathbb{R}\times\mathbb{R}^{d}\rightarrow\mathbb{C}$, $m\in (0,+\infty)$, and the nonlinearity is $H^1$-subcritical, i.e. $1<p<1+\frac{4}{d-2}$ if $d\geq 3$ or $1<p<+\infty$ if $d=1,2$.

This equation arises in particular in  Quantum Physics  where it has been proposed as a simple model describing a self interacting scalar field. 
Mathematically speaking, the Klein-Gordon equation is one of the model dispersive equations. It is a Hamiltonian equation which is invariant under gauge and Lorentz transform and in particular it conserves energy, charge and momentum.  Due to the sign of the nonlinearity, the equation is focusing. At the balance  between dispersion and focusing, we find ``truly'' nonlinear solutions: the stationary/standing/solitary  waves. 

A \emph{standing wave} with frequency $\omega\in\mathbb R$ is a solution of~\eqref{eq:nlkg} of the form $u(t,x)=e^{i\omega t}\varphi_\omega(x)$. Such solution has the particularity to exist globally  and to remain localized at any time.  In the physics literature this kind of solutions are sometimes referred to as \emph{Q-balls}.
A \emph{soliton} (or \emph{solitary wave})  with speed $v\in\R^d$, frequency $\omega\in\mathbb R$ and initial phase and position $\theta\in\R$, $x_0\in\R^d$ is a  boosted standing wave solution of~\eqref{eq:nlkg}.
More precisely  a soliton is a  solution of~\eqref{eq:nlkg} of the form 
\[
e^{i\frac{\omega}{\gamma}t+i\theta}\varphi_{\omega,v}(x-vt-x_0),
\]
 where $\varphi_{\omega,v}$ is a  profile depending on $\omega$ and $v$, and  $\gamma:=\frac{1}{\sqrt{1-|v|^2}}$ is the Lorentz factor.

We shall consider ground states solitons, i.e boosted standing waves with ground states profiles (profiles minimizing a certain action functional, see Section~\ref{sec:preliminaries} for a more precise definition). The orbital stability properties of such solitons have been widely studied.
It started with the work of Shatah~\cite{Sh83} where it was shown that there exists a critical frequency $\omega_c>0$ such that if $p<1+4/d$ and  $\omega_c<|\omega|<m$ then standing waves are stable under radial perturbation. Later on, Shatah~\cite{Sh85} proved that the stationary solution (i.e. the standing wave with $\omega=0$) is strongly unstable and the picture for standing waves was completed by Shatah and Strauss~\cite{ShSt85} when they showed that if either $p\geq1+4/d$ or if $|\omega|<\omega_c$ then standing waves are unstable. These results were generalized and consolidated by Grillakis, Shatah and Strauss in their celebrated works~\cite{GrShSt87,GrShSt90}.  
The stability theory of solitons was revisited by Stuart~\cite{St01} via the modulational approach introduced by Weinstein~\cite{We85} for nonlinear Schr\"odinger equations. Compare to prior results,  Stuart~\cite{St01} provided two improvements: first, he treated the whole range of possible speeds $|v|<1$ without the radiality assumptions, second he gave the laws of the modulations parameters. In particular, it was shown in~\cite{St01} that the ground state solitons are stable if the parameters are within the following open set
\begin{equation}\label{eq:O-stab}
\mathcal O_{\mathrm{stab}}:=\Bigg\{
(\omega,\theta,v,x)\in\R^{2+2d};\;
|\omega|<\sqrt{m},\,|v|<1,\,\frac{1}{1+\frac{4}{p-1}-d}<\frac{\omega^2}{m}
\Bigg \}.
\end{equation}
Note that $\mathcal O_{\mathrm{stab}}$ is nonempty only if $p<1+\frac{4}{d}$, i.e. the nonlinearity is $L^2$-subcritical. 
Instability was further investigated by Liu, Ohta and Todorova~\cite{LiOhTo07,OhTo05,OhTo07} (see also~\cite{JeLe09} for a companion result), who proved that when standing waves are unstable, then the instability is either strong (i.e. by blow up in possibly infinite time) when $p<1+4/d$ or very strong (i.e. by blow up in finite time) when $p\geq1+4/d$. 

Recently,  
further informations on the dynamics of \eqref{eq:nlkg} around solitons have been obtained by Nakanishi and Schlag.
In~\cite{NaSc12-2}, using a method refered to as Hadamard approach in dynamical systems, they show the existence of a center-stable manifold which contains all solutions of \eqref{eq:nlkg}  staying close to the solitons manifold and describe precisely this manifold. Furthermore, in~\cite{NaSc11-2}, they adopt a Lyapunov-Perron approach for the study of the dynamics around ground state stationary solitons of \eqref{eq:nlkg} for the 3-d cubic case and in a radial setting. In particular, they show the existence of a center stable manifold such that the following trichotomy occurs for a solution with inital data close to the ground state stationary solution. On one side of the center stable manifold, the solution scatters to $0$, on the other side it blows up in finite time and on the center stable manifold itself the solution scatters to the ground state. The same authors~\cite{NaSc12-1} extended later their results in the non-radial setting.
 One can also refer to their monograph~\cite{NaSc11} for a complete introduction to the mathematical study of equations similar to \eqref{eq:nlkg}, in particular the study of the dynamics of the equation around stationary/standing waves.

In this paper we address the question whether it is possible to construct a multi-soliton solution for~\eqref{eq:nlkg}, i.e. a solution behaving at large time like a sum of solitons. Multi-solitons are long time known to exist for integrable equations such as the  Korteweg-de Vries equation or the 1-d cubic nonlinear Schr\"odinger equation. Indeed, existence of multi-solitons follows from the inverse scattering transform,  see e.g. the survey of Miura~\cite{Mi76} for the Korteweg-de Vries equation and the work of Zakharov and Shabat~\cite{ZaSh72} for the 1-d cubic nonlinear Schr\"odinger equation. In the recent years, there has been a series of works around the existence and dynamical properties of multi-solitons for various dispersive equations. 

One of the first existence result of multi-solitons for non-integrable equations was obtained by Merle~\cite{Me90} as a by-product of the construction of a multiple blow-up points solution to the $L^2$-critical nonlinear Schr\"odinger equation (indeed a pseudo-conformal transform of this solution gives the multi-soliton). Later on, Perelman~\cite{Pe97,Pe04}  (see also \cite{RoScSo03}) studied asymptotic stability of a sum of solitons of nonlinear Schr\"odinger equation under spectral hypotheses and in weighted spaces. In the energy space, Martel, Merle and Tsai~\cite{MaMe06,MaMeTs06} showed the existence and orbital stability of multi-solitons made of stable solitons. The existence of multi-solitons made of unstable solitons was obtained by C\^ote, Martel and Merle~\cite{CoMaMe11} for ground state and by C\^ote and Le Coz~\cite{CoLe11} for excited states under a high speed assumption. Further results on the existence of exotic solutions like a train of infinitely many solitons were obtained by Le Coz, Li and Tsai \cite{LeLiTs13}. 

For the non-integrable generalized Korteweg-de Vries equation, Martel~\cite{Ma05-1} showed the existence and uniqueness of multi-solitons for $L^2$-subcritical nonlinearities. These multi-solitons were shown to be stable and asymptotically stable by Martel, Merle and Tsai~\cite{MaMeTs02}. Combet~\cite{Co11} investigated further the existence of multi-solitons in the supercritical case and showed the existence and uniqueness of a $N$-parameter family of multi-solitons. Outstanding results on the description of the interaction between two solitons were recently obtained by Martel and Merle~\cite{MaMe11-2,MaMe11-3,MaMe11}.

Despite the many works  on multi-solitons previously cited, to our knowledge the present paper and the recent preprint~\cite{CoMu12} are the first works dealing with existence of multi-soliton type solutions for nonlinear Klein-Gordon equations (see nevertheless~\cite{CoZa13} for related results on the nonlinear wave equation).

Our goal is to prove the following existence result for multi-solitons.
\begin{theorem}\label{thm:1}
Assume that $1<p<1+\frac{4}{d}$.
For any  $N\in\mathbb N$, take 
$(\omega_j,\theta_j,v_j,x_j)_{j=1,\dots, N}\subset\mathcal O_{\mathrm{stab}}$ and let $(\varphi_j)$ be the associated ground state profiles $\varphi_j:=\varphi_{\omega_j,v_j}$, and $(\gamma_j)$ the  Lorentz factors $\gamma_j:=(1-|v_j|^2)^{-\frac{1}{2}}$. Denote the corresponding solitons by 
\[
\mathcal R_j(t,x):=e^{i\frac{\omega_j}{\gamma_j}t+i\theta_j}\varphi_j(x-v_jt-x_j).
\]
Define
\vspace{-5pt}
\begin{align}
v_\star&:=\min\{|v_j-v_k|,j,k=1,\dots ,N,j\neq k\}&&\text{({minimal relative speed})}\label{eq:v-star}\\
\omega_\star&:=\max\{|\omega_j|,j=1,\dots ,N\}&&\text{({maximal frequency})}\label{eq:omega-star}
\end{align}
There exists $\alpha=\alpha(d,N)>0$, such that if $v_j\neq v_k$ for any $j\neq k$,  then there exist $T_0\in\R$ and a solution $u$ to~\eqref{eq:nlkg} existing on $[T_0,+\infty)$ and such that  for all $t\in[T_0,+\infty)$  the following estimate holds
\[
\norm[\bigg]{u(t)-\sum_{j=1}^N\mathcal R_j(t)}_{H^1}+\norm[\bigg]{\partial_tu(t)-\sum_{j=1}^N\partial_t\mathcal R_j(t)}_{L^2}\leq e^{-\alpha\sqrt{m-\omega^2_\star}v_\star t}.
\]
\end{theorem}

\begin{remark}
During the preparation of this paper we have been aware of the work~\cite{CoMu12} by C\^ote and Mu\~noz. Our two results are companions in the following sense. In~\cite{CoMu12} the authors  used spectral theory and a topological argument to prove the existence of  multi-solitons made of unstable solitons. To the contrary, we use finite speed of propagation, classical modulation theory and energy estimates to obtain the existence of multi-solitary waves based on stable solitons. Merging our results together would give the existence of multi-solitons made with any kind (stable or unstable) of solitons. 
\end{remark}


Our strategy for the proof of Theorem~\ref{thm:1} is to slove  \eqref{eq:nlkg} backwards around suitable approximate solutions. It is inspired by the works of Martel, Merle and Tsai  on multi-solitons of Schr\"odinger equations~\cite{MaMe06, MaMeTs06} (see also~\cite{CoLe11,CoMaMe11} where similar strategies were enforced). The main new ingredients on which we rely are the variational characterizations of the profile and a coercivity property of the total linearized action.

We start  by introducing the mathematical framework in which we are going to work in Section~\ref{sec:preliminaries}. After transforming~\eqref{eq:nlkg} into its Hamiltonian form~\eqref{eq:nlkg2}, we list the tools which are going to be useful for our purposes: Cauchy Theory in $H^1\times L^2$ and $H^s\times H^{s-1}$, Conservation laws,  Finite Propagation Speed, standing waves, Lorentz transform and finally definitions of solitons and their profiles.

Then we go on with the core of the proof of Theorem~\ref{thm:1}. We consider a sequence of times $T^n\rightarrow \infty$, a set of final data $u_n=\sum\mathcal R_j(T^n)$ and the associated solutions $(u_n)$ of~\eqref{eq:nlkg} backward in time. The sequence $(u_n)$ provides us with a sequence of approximate multi-solitons, and we need to prove its convergence  to a solution of~\eqref{eq:nlkg} satisfying to the conclusion of Theorem~\ref{thm:1}. For this purpose, we show that  each $u_n$ exists backwards in time up to some time $T_0$ independent of $n$ and decay uniformly in $n$ to the sum of solitons (Proposition~\ref{prop:uniform}). Eventually a compactness argument (Lemma~\ref{lem:compact}) permits to show that $(u_n)$ converges to a multi-soliton of~\eqref{eq:nlkg} on $[T_0, \infty)$. Most steps are performed in Section~\ref{sec:proof}, apart from  uniform estimates whose proof needs more preparation.

The proof of the uniform estimates relies on several ingredients:  coercivity of the Hessian of the action around each component of the multi-soliton, modulation theory and slow variation of localized conservation laws, energy, charge, momenta.

We study the profiles of the solitons in Section~\ref{sec:variational}.
We characterize the profiles variationally using the conserved quantities of~\eqref{eq:nlkg} (Proposition~\ref{prop:variational}). 
and show that 
the ground state profiles are at the mountain pass level, the least energy level and the Nehari level. 
Our proofs are self-contained and do not rely on the (\textsc{nls}) case. 

After obtaining the variational characterizations, we prove a coercivity property (Lemma~\ref{lem:coercivity}) for the second variation of the action functional around a soliton (linearized action). To this aim, we study the spectrum of the linearized action and prove in particular the non degeneracy of the kernel (i.e. the kernel contains only the eigenvectors generated by the invariances of the equation, see Lemma~\ref{lem:kernel}). It is usually a crucial point in these matters
We underline that the coercivity properties are related to the fact that our standing waves are stable.

Coercivity of the linearized action is obtained provided orthogonality conditions hold. This prompt the question of obtaining orthogonality conditions around a sum of solitons, which is resolved in Section~\ref{sec:modulation} via modulation theory. The modulation result is twofold. First, it shows that close to a sum of solitons one can recover orthogonality conditions (see~\eqref{eq:ortho-dyn}) by adjusting the modulation parameters phases, translation and scaling. Second, it gives the dynamical laws followed by the  parameters (see~\eqref{eq:parameters-derivatives-estimate}). 

Finally, we define cutoff functions around each soliton and use them to localize the action around each soliton. We use these localized actions to build a global action adapted to the sum of solitons. Several properties are transported from the local actions to the global one. In particular, the global action inherits from the coercivity (see Lemma~\ref{lem:multi-coercivity}). Due to errors generated by the cutoff it is not a conserved quantity, but we can however prove that it is almost conserved (i.e. it varies slowly). We use  these properties combined with the modulation result to prove the uniform estimates.

This work is organized as follows. In Section~\ref{sec:preliminaries}, we set the mathematical work context. Section~\ref{sec:proof} contains the proof of the main result, assuming uniform estimates. In Section~\ref{sec:variational}, we establish variational characterizations of the profiles and use them to prove a coercivity statement for the hessian of the action functional related to a soliton. In Section~\ref{sec:modulation}, we explain the modulation theory in the neighborhood of a sum of solitons. Finally, we put all pieces together in Section~\ref{sec:uniform} to prove the uniform estimates. The Appendix contains the proof of a compact injection used in Section~\ref{sec:proof} and interactions estimates used in Section~\ref{sec:uniform}.

\section{Mathematical context}\label{sec:preliminaries}

In this section we introduce rigorously all the necessary material for our study and restate our result in the Hamiltonian formulation for~\eqref{eq:nlkg}, which is a more suitable formulation for our needs. But before  let us precise some notations.  We denote by $L^q(\mathbb R^d)$ the standard Lebesgue space and its norm by $\norm{\cdot}_q$. The space $L^2(\R^d)$ is viewed as a real Hilbert space endowed with the scalar product
\[
\psld{u}{v}=\Re\int_{\R^d}u\bar vdx.
\]
The Sobolev spaces $H^s(\R^d)$ are endowed with their usual norms $\norm{\cdot}_{H^s}$. For the product space $L^2(\R^d)\times L^2(\R^d)$ we use the norm
\[
\norm[\bigg]{\binom{u_1}{u_2}}_{L^2\times L^2}=\sqrt{\norm{u_1}^2_2+\norm{u_2}_2^2},
\]
with similar convention for $H^1(\R^d)\times L^2(\R^d)$, $H^s(\R^d)\times H^{s-1}(\R^d)$, etc.
We shall sometimes us the following notational shortcut:
\[
\psld{u}{\nabla v}:=\begin{pmatrix}\psld{u}{\frac{\partial v}{\partial x_1}}\\\vdots\\\psld{u}{\frac{\partial v}{\partial x_d}}\end{pmatrix}.
\]
Finally, unless otherwise specified the components of a vector $W\in L^2(\R^d)\times L^2(\R^d)$ will be denoted by $W_1$ and $W_2$.

The \emph{Hamiltonian Formalism} for the nonlinear Klein-Gordon equation~\eqref{eq:nlkg} is formulated  as follows. For $(u_{1},u_{2})\in H^{1}(\R^d)\times L^{2}(\R^d)$ we
define the following Hamiltonian (which we will call \emph{energy} in the sequel)
\begin{equation*}
E\binom{u_1}{u_2} 
	:=  \frac{1}{2}\norm{ u_2}_2^2+\frac12\norm{\nabla u_1}_2^2+\frac{m}{2}\norm{u_1}_2^2-\frac{1}{p+1}\norm{u_1}_{p+1}^{p+1}.
\end{equation*}
Define the matrix $J:=\begin{pmatrix}
0 & 1\\
-1 & 0
\end{pmatrix}$. 
Then $u$ is a solution of~\eqref{eq:nlkg} if and only if 
$(u_{1},u_{2}):=(u,u_{t})$ 
solves
 the following equation
\begin{equation}\label{eq:nlkg2}
\partial_{t}\binom{u_{1}}{u_{2}}=JE'\binom{u_{1}}{u_{2}}.
\end{equation}
From now on we shall work only with the Hamiltonian equation~\eqref{eq:nlkg2}.

Due to this Hamiltonian formulation the energy is (at least formally) conserved. In addition, the invariance of~\eqref{eq:nlkg2} under phase shifts and space translations generates two other \emph{conservations laws}, the \emph{charge} $Q$ and the (vectorial) \emph{momentum} $P$, defined in the following way:
\begin{equation*}
Q\binom{u_{1}}{u_{2}} = \Im\int_{\R^d} u_{1}\bar{u}_{2}dx,\qquad
P\binom{u_{1}}{u_{2}}=\Re\int_{\R^d}\nabla u_{1}\bar{u}_{2}dx.
\end{equation*}

With our restrictions on the growth of the nonlinearity in Theorem~\ref{thm:1} ($L^2$-subcritical), it is well-known that the Cauchy problem for~\eqref{eq:nlkg2} is globally well-posedness in the energy space $H^1(\R^d)\times L^2(\R^d)$. More precisely, the following well-posedness theory holds.

\begin{lwp1}
Assume $1<p<1+\frac4d$. For any initial data $U_0=(u_1^0,u_2^0)\in H^1(\R^d)\times L^2(\R^d)$ there exists a unique maximal  solution of~\eqref{eq:nlkg2}
\[
U\in \mathcal C\left(\R,H^1(\R^d)\times L^2(\R^d)\right)\cap \mathcal C^1\left(\R,L^2(\R^d)\times H^{-1}(\R^d)\right).
\]
Furthermore, we have the following properties.\\
\emph{Conservation of energy, charge and momentum}: for all $t\in\R$, we have
\[
E(U(t))=E(U_0),\quad Q(U(t))=Q(U_0),\quad P(U(t))=P(U_0),
\]
\emph{Global estimate}: there exist $C_0>0$ such that $\norm{U}_{\mathcal C(\R,H^1\times L^2)}\leq C_0\norm{U_0}_{H^1\times L^2}$.\\
\emph{Uniqueness in light cones}: If $\tilde U$ is another solution to~\eqref{eq:nlkg2} on $(0,T)$ for $T>0$ with $\tilde U(0)=U_0$ in $\{x:\,|x-x_0|<T\}$ for some $x_0\in\R^d$, then $\tilde U\equiv U$ on the backward light cone $\{(t,x)\in[0,T)\times \R^d\;:|x-x_0|<T-t\}$.\\
\emph{Continuous dependency upon the initial data}:  if $(U_0^n)\subset H^1(\R^d)\times L^2(\R^d)$ converges to $U_0$ in $H^1(\R^d)\times L^2(\R^d)$, then the associated solutions $(U_n)$ of~\eqref{eq:nlkg2} converge in $\mathcal C\left(I,H^1(\R^d)\times L^2(\R^d)\right)$ for any compact time interval $I\subset \R$ to the solution $U$ of~\eqref{eq:nlkg2} with initial data $U(0)=U_0$.
\end{lwp1}

For this set of results, we refer to 
the classical papers by Ginibre and Velo~\cite{GiVe85,GiVe89}, or the recent review in the  paper~\cite{KiStVi12} by Killip, Stovall and Visan.
For our purposes, we will also need a more refined result on local well-posedness in the slightly larger space $H^s(\R^d)\times H^{s-1}(\R^d)$ for some $s<1$ (see Lindblad and Sogge~\cite{LiSo95} or  Nakamura and Ozawa~\cite{NaOz01}).

\begin{lwp2}
Let $s>0$ be such that either $s>d/2$ or $1/2\leq s<d/2$ and $p< 1+\frac{4}{d-2s}$. For any initial data $U_0=(u_1^0,u_2^0)\in H^s(\R^d)\times H^{s-1}(\R^d)$ there exists a unique maximal  solution of~\eqref{eq:nlkg2}
\[
U\in \mathcal C\left((-T_\star,T^\star),H^s(\R^d)\times H^{s-1}(\R^d)\right).
\]
Furthermore, we have the \emph{continuous dependent upon the initial data}:  if $(U_0^n)\subset H^s(\R^d)\times H^{s-1}(\R^d)$ converges to $U_0$ in  $H^s(\R^d)\times H^{s-1}(\R^d)$ , then the associated solutions $(U_n)$ of~\eqref{eq:nlkg2} converge to $U$ in $\mathcal C\left(I,H^s(\R^d)\times H^{s-1}(\R^d)\right)$ for any compact time interval $I\subset (-T_\star,T^\star)$, where $U$ is the solution to~\eqref{eq:nlkg2} with initial data $U(0)=U_0$.
\end{lwp2}

A useful consequence of the uniqueness in light cones (Cauchy Theory in $H^1\times L^2$) is the following \emph{finite speed of propagation} property.
\begin{FPS}
Let $U=(u_1,u_2)\in H^1(\R^d)\times L^2(\R^d)$ be a solution of~\eqref{eq:nlkg2} on $(-\infty,T^\star]$. There exists $C_0$, depending only on $\norm{U(T^\star)}_{H^1\times L^2}$  such that if there exist $0<\varepsilon$ and $M>0$ satisfying 
\[
\int_{|x|>M}|\nabla u_1(T^\star)|^2+|u_1(T^\star)|^2+|u_2(T^\star)|^2dx\leq \varepsilon,
\]
then for any $t\in [-\infty,T^\star]$ we have
\[
\int_{|x|>2M+(T^\star-t)}|\nabla u_1(t)|^2+|u_1(t)|^2+|u_2(t)|^2dx\leq C_0\varepsilon,
\]	
\end{FPS}

\begin{proof}
Let $\chi_M$ be a cutoff function such that 
\[
\chi_M(x)=\begin{cases}
1&\text{ for }|x|>2M\\
0&\text{ for }|x|<M 
\end{cases}
\quad\text{ and }\quad
\norm{\nabla\chi_M}_\infty<\frac{C_0}{M}.
\]
Define $U_{T^\star,M}:=U(T^\star)\chi_M$ and denote by $U_M$ the associated solution of~\eqref{eq:nlkg2}. By assumption, we have $\norm{U_{T^\star,M}}^2_{H^1\times L^2}\leq \varepsilon$ and by the Cauchy Theory in $H^1\times L^2$ the solution $U_M$ exists on $\R$ and verifies for all $t\in\R$
\[
\norm{U_M(t)}^2_{H^1\times L^2}\leq C_0\varepsilon.
\]
However, by uniqueness on light cones, $U_M$ and $U$ coincide on $\{(t,x)\in \R\times\R^d:\,|x|>2M+(T^\star-t)\}$, and  for any $t\in (-\infty,T^\star)$ this implies
\[
\int_{|x|>2M+(T^\star-t)}|\nabla u_1(t)|^2+|u_1(t)|^2+|u_2(t)|^2dx\leq \norm{U_M(t)}^2_{H^1\times L^2}\leq C_0\varepsilon,
\]	
which was the desired conclusion.
\end{proof}

\paragraph{Lorentz transform.} Among the symmetries of~\eqref{eq:nlkg2}, we already mentioned the phase shift and translation. We consider now the Lorentzian symmetry, defined as follows.
Take $U(t,x)=(u_1,u_2)(t,x)$ and  $v\in \R^d$ with $|v|$ smaller than the speed of light for~\eqref{eq:nlkg2}, namely $|v|<1$. The \emph{Lorentz transform}  $\mathscr{L}_vU$ of $U$ is the function of $(t,x)$ defined by
\[
(\mathscr{L}_vU)(t,x):=\begin{pmatrix}u_1(\tau,y)\\
\gamma\left( u_2(\tau,y)- v\nabla_y u_1(\tau,y) \right)
\end{pmatrix}
\]
where $\tau$ and $y$ are defined by
 \begin{align*}
\tau&=\tau(t,x):=\gamma(t-v\cdot x)=\frac{1}{\gamma}t-\gamma (x-vt)\cdot v,\\
y&=y(t,x):=x-x_v+\gamma(x_v-vt)=x-vt+(\gamma-1)(x-vt)_v.
\end{align*}
Here, the Lorentz parameter $\gamma$ is defined by
\[
\gamma:=\frac{1}{\sqrt{1-|v|^2}},
\]
 and the subscript $v$ denote the orthogonal projection onto the vectorial line generated by $v$, that is
\[ 
x_v:=\frac{x\cdot v}{|v|^2}v.
\]
It is simple algebra to verify that~\eqref{eq:nlkg2} is Lorentz invariant, in the sense that if $U$ is a solution of~\eqref{eq:nlkg2}, then  so is $\mathscr{L}_vU$. Also note that the Lorentz transform is invertible with inverse $\mathscr{L}_{-v}$.

\paragraph{Standing waves.} 
Take $\omega\in\mathbb{R}$. In the Hamiltonian formulation,  a \emph{standing wave} with
frequency $\omega$ is a solution of~\eqref{eq:nlkg2} of the form $U(t,x)=e^{i\omega t}\Phi_\omega(x)$. 
Plugging this ansatz for $U$ into~\eqref{eq:nlkg2}, it is easy to see that  $\Phi_\omega=\binom{\varphi_{\omega,1}}{\varphi_{\omega,2}}$  must  be a critical point of $E+\omega Q$, hence a solution to the stationary elliptic system
\begin{equation*}
\left\{
\begin{aligned}
-\Delta \varphi_{\omega,1} +m\varphi_{\omega,1}-|\varphi_{\omega,1}|^{p-1}\varphi_{\omega,1}+i\omega\varphi_{\omega,2}&=0,\\
\varphi_{\omega,2}-i\omega\varphi_{\omega,1}&=0.
\end{aligned}
\right.
\end{equation*}
The solutions of this system are clearly of the form $\binom{\varphi_\omega}{i\omega\varphi_{\omega}}$, where $\varphi_\omega$ satisfies the scalar equation
\begin{equation}\label{eq:scalar-snlkg}
-\Delta \varphi_\omega +(m-\omega^2)\varphi_\omega-|\varphi_\omega|^{p-1}\varphi_\omega=0.
\end{equation}
Solutions to~\eqref{eq:scalar-snlkg} and their properties are well-known (see~\cite{BeLi83-1,BeLi83-2,GiNiNi79,Kw89} and the references therein). 
For every $\omega\in(-\sqrt{m},\sqrt{m})$ there exists a unique, positive, and radial function $\varphi_{\omega}\in\mathcal{C}^2(\R^d)$  solution  of~\eqref{eq:scalar-snlkg}. 
In addition, 
the function $\varphi_{\omega}$ is exponentially decaying at infinity: for any $\mu<(m-\omega^2)$ there exists $C(\mu,\omega)>0$ such that
\begin{equation}\label{eq:exponential-decay}
|\varphi_{\omega}(x)|\leq C(\mu,\omega)e^{-\sqrt{\mu}|x|}\qquad\text{for all }x\in\R^d. 
\end{equation}
Furthermore, any $\varphi_{\omega}$ satisfy the scaling property 
\begin{equation}\label{eq:scaled-phi} 
\varphi_{\omega}(x)=\left({m-\omega^2}\right)^{\frac{1}{p-1}} \tilde\varphi\left(\left({m-\omega^2}\right)^{\frac12}x\right),
\end{equation}
where $ \tilde\varphi$ is the unique positive radial solution to $-\Delta \tilde \varphi+\tilde\varphi-|\tilde\varphi|^{p-1}\tilde\varphi=0$.
The function $\varphi_\omega$ is called \emph{ground state}. In dimension $d\geq 2$, there exist infinitely many other solutions to~\eqref{eq:scalar-snlkg}, called \emph{excited states}.
In the sequel, we shall deal only with ground states solutions to~\eqref{eq:scalar-snlkg}. Indeed, our analysis deeply relies on properties of the ground states which do not hold for other solutions, in particular the stability of the associated standing waves (see Section~\ref{sec:variational} for details).
\begin{remark}
It is interesting to notice that, although the presence of the nonlinear term permits the existence of states with negative energy, the standing waves have always positive energies. Indeed, a straightforward computation assures that for a standing wave $U(t,x)=e^{i\omega t}\Phi_\omega(x)$ the corresponding energy is given by
$$ 
E(\Phi_\omega)=\left(\frac{1}{2}-\frac{1}{p+1}\right)\left(\norm{\nabla \varphi_{\omega}}_2^2+m\norm{\varphi_{\omega}}_2^2\right) +\bigg(\frac 12+\frac{1}{p+1}\bigg)\omega^2\norm{\varphi_{\omega}}_2^2
$$ 
The fact that $p>1$ guarantees that $E(\Phi_\omega)>0.$
\end{remark}

\begin{remark}
The scaling property  \eqref{eq:scaled-phi} guarantees that the energy of the ground states varies continuously with respect to $\omega$. This fact implies that
 the  multi-soliton solutions for~\eqref{eq:nlkg} behave at large time like a sum of solitons that are allowed to have different energies. A straightforward computation indeed gives
\begin{align*}
&\norm{\varphi_{\omega}}_2^2=(m-\omega^2)^{\frac{4-d(p-1)}{2(p-1)}}\norm{\tilde\varphi}_2^2\\
&\norm{\nabla \varphi_{\omega}}_2^2=(m-\omega^2)^{\frac{p(2-d)+2+d}{2(p-1)}}\norm{\nabla \tilde\varphi}_2^2.
\end{align*}
Now $ \tilde\varphi$ is solution of  $-\Delta \tilde \varphi+\tilde\varphi-|\tilde\varphi|^{p-1}\tilde\varphi=0$ such that, by means of Pohozaev identity, 
$$\norm{\nabla \tilde\varphi}_2^2=\frac{d(p-1)}{2d-(d-2)(p+1)}\norm{\tilde\varphi}_2^2.$$
Merging all this information we get the relation between energy and $\omega$ given by
$$E(\varphi_\omega)=g(\omega)\norm{\tilde\varphi}_2^2,$$
where
\[
g(\omega)=\frac{p-1}{2(p+1)}\left( (m-\omega^2)^{\frac{p(2-d)+2+d}{2(p-1)}}+m(m-\omega^2)^{\frac{4-d(p-1)}{2(p-1)}}\right)
+\left( \frac{p+3}{2(p+1)}\omega^2(m-\omega^2)^{\frac{4-d(p-1)}{2(p-1)}}\right),
\]
which can be rewritten as
$$
g(\omega)=\frac{(m (p-1)+2 \omega^2) (m-\omega^2)^{\frac2{p-1}-\frac d2}}{(p+1)}.
$$
The monotonicity of $g( \omega)$ when $\omega$ belongs to $\mathcal O_{\rm stab}$
follows easily.
\end{remark}

\paragraph{Solitons.}
Starting from a standing wave, one generates a  new family of solutions to~\eqref{eq:nlkg2} simply by boosting them using the Lorentz transform. These new solutions are the \emph{solitary waves} (or simply \emph{solitons}). 
Precisely, take a frequency $|\omega|\leq\sqrt{m}$, the profile $\Phi_\omega:=\binom{\varphi_\omega}{i\omega\varphi_\omega}$ (where $\varphi_\omega$ is the ground state of~\eqref{eq:scalar-snlkg}), a phase $\theta\in\mathbb{R}$, and a speed and a position $v,x_{0}\in\mathbb{R}^{d}$  with
$|v|<1$. The associated soliton is 
\[
e^{i\frac{\omega}{\gamma}t+i\theta}\Phi_{\omega,v}(x-vt-x_0),
\]
where the new profile $\Phi_{\omega,v}$ is given by
 \begin{equation}\label{eq:profile}
\Phi_{\omega,v}(x)=e^{-i\gamma\omega v\cdot x}\binom{\varphi_{\omega}(x+(\gamma-1)x_{v})}{\gamma(i\omega\varphi_\omega(x+(\gamma-1)x_{v})-v\nabla\varphi_{\omega}(x+(\gamma-1)x_{v}))}.
\end{equation}
By direct computation, and provided we have noticed that
\begin{gather*}
E'\binom{u_{1}}{u_{2}} =  \begin{pmatrix}
-\Delta u_{1}+mu_{1}-|u_{1}|^{p-1}u_{1}\\
u_{2}
\end{pmatrix},\\
Q'\binom{u_{1}}{u_{2}} =  \begin{pmatrix}
iu_{2}\\
-iu_{1}
\end{pmatrix}
=iJ\binom{u_{1}}{u_{2}},
\qquad
P'\binom{u_{1}}{u_{2}} = \begin{pmatrix}
-\nabla u_{2}\\
\nabla u_{1}
\end{pmatrix}
=-J\nabla\binom{u_{1}}{u_{2}},
\end{gather*}
it is not difficult to see that $\Phi_{\omega,v}$ is a critical point of 
\[
S:=E+\frac{\omega}{\gamma}Q+v\cdot P.
\]

With all these preliminaries out of the way, we can go on with the proof of Theorem~\ref{thm:1}.

\section{Existence of Multi-Solitons}\label{sec:proof}

This section contains the core of the proof of Theorem~\ref{thm:1} assuming uniform estimates (Proposition~\ref{prop:uniform}) which are proved in Section~\ref{sec:uniform}. 

Assume that $p<1+\frac{4}{d}$. Take $N\in\mathbb N$, $(\omega_j,\theta_j,v_j,x_j)_{j=1,\dots, N}\subset\mathcal O_{\mathrm{stab}}$, $\Phi_j$  the associated Hamiltonian profiles (as in~\eqref{eq:profile}), $v_\star$ and $\omega_\star$ as in~\eqref{eq:v-star},~\eqref{eq:omega-star},  $(\gamma_j)$ the Lorentz parameters,  $(R_j)$ the corresponding solitons 
\[
R_j(t,x):=e^{i\frac{\omega_j}{\gamma_j}t+i\theta_j}\Phi_j(x-v_jt-x_j),
\]
 and $R$ the sum of the solitons :
\[
R(t,x):=\sum_{j=1}^N R_j(t,x).
\]
Reformulated using the Hamiltonian expression~\eqref{eq:nlkg2} of~\eqref{eq:nlkg}, our goal is to prove that there exists $\alpha=\alpha(d,N)>0$, such that if $v_j\neq v_k$ for any $j\neq k$,  then there exist $T_0\in\R$ and a solution $U$ to~\eqref{eq:nlkg2} existing on $[T_0,+\infty)$ and such that   the following estimate holds for all $t\in[T_0,+\infty)$
\[
\norm*{U(t)-R(t)}_{H^1\times L^2}\leq e^{-\alpha\sqrt{m-\omega^2_\star}v_\star t}.
\]

We are going to define a sequence of approximate multi-solitons and prove its convergence to the desired solution of~\eqref{eq:nlkg2}.
Take an increasing sequence of time $T^n\to+\infty$ and for each $n$ let $U_n$ be the solution to~\eqref{eq:nlkg2} obtained by solving~\eqref{eq:nlkg2} backward in time from $T^n$ with final data  $U_n(T^n)=R(T^n)$. Our proof will rely  on two main ingredients. First we have uniform estimates for the sequence of approximate multi-solitons.

\begin{proposition}[Uniform Estimates]\label{prop:uniform}
There exist $\alpha=\alpha(d,N)>0$,  and $T_0\in\R$ (independent of $n$) such that for $n$ large enough the solution $U_n$ of~\eqref{eq:nlkg2} with $U_n(T^n)=R(T^n)$ exists on $[T_0,T^n]$ and satisfies for all $t\in[T_0,T^n]$ the estimate
\begin{equation}\label{eq:uniform}
\norm{U_n(t)-R(t)}_{H^1\times L^2}\leq e^{-\alpha \sqrt{m-\omega_\star^2} v_\star t}
\end{equation}
\end{proposition}

Proposition~\ref{prop:uniform} establishes that the approximate multi-solitons $U_n$ all satisfy the desired estimate on time intervals of the form $[T_0,T^n]$, with $T_0$ \emph{independent} of $n$. 
The proof of Proposition~\ref{prop:uniform} is rather involved and we postpone it to Section~\ref{sec:uniform} (useful informations for this proof are derived in Sections~\ref{sec:variational} and~\ref{sec:modulation}).  

The second ingredient of the proof of Theorem~\ref{thm:1} is an $H^1\times L^2-$compactness property of the sequence of initial data of the approximate multi-solitons.

\begin{lemma}[Compactness]\label{lem:compact}
Let $T_0$ be given by Proposition~\ref{prop:uniform}. For any $\varepsilon>0$ there exists $M_\varepsilon$ such that for any $n$ large enough $U_n$ verifies
\[
\int_{|x|>M_\varepsilon}|\nabla U_{n,1}(T_0)|^2+|U_{n,1}(T_0)|^2+|U_{n,2}(T_0)|^2dx\leq \varepsilon.
\]
\end{lemma}

The argument for the proof of Lemma~\ref{lem:compact} is different from the Schr\"odinger equation case.  Indeed, we benefit with  the Klein-Gordon equation of the Finite Propagation Speed, which is not the case for Schr\"odinger equations where one has to us virial identities (see e.g.~\cite[Lemma 2]{MaMe06}).

\begin{proof}[Proof of Lemma~\ref{lem:compact}]
The result is a consequence of the Finite Speed of Propagation and the uniform estimates of Proposition~\ref{prop:uniform}. Indeed, take $\varepsilon>0$ and let $T^\star$ be such that $e^{-\alpha \sqrt{m-\omega_\star^2} v_\star T^\star}<\frac{\varepsilon}{2}$. Then it follows from Proposition~\ref{prop:uniform} that for $n$ large enough
\begin{equation}\label{eq:compact-1}
\norm{U_n(T^\star)-R(T^\star)}_{H^1\times L^2}\leq \frac{\varepsilon}{2}.
\end{equation}
By exponential decay of the sum of solitons, there exists $\tilde M_\varepsilon$ such that 
\begin{equation}\label{eq:compact-2}
\int_{|x|>\tilde M_\varepsilon}|\nabla (R(T^\star)_1)|^2+|(R(T^\star))_1|^2+|(R(T^\star))_2|^2dx\leq \frac{\varepsilon}{2}.
\end{equation}
Combining~\eqref{eq:compact-1} and~\eqref{eq:compact-2}, we get  
\[
\int_{|x|>\tilde M_\varepsilon}|\nabla U_{n,1}(T^\star)|^2+|U_{n,1}(T^\star)|^2+|U_{n,2}(T^\star)|^2dx\leq \varepsilon.
\]
By  Finite Speed of Propagation, this implies
\[
\int_{|x|>\tilde 2M_\varepsilon+(T^\star-T_0)}|\nabla U_{n,1}(T_0)|^2+|U_{n,1}(T_0)|^2+|U_{n,2}(T_0)|^2dx\leq \varepsilon.
\]
Setting $M_\varepsilon=2\tilde M_\varepsilon+(T^\star-T_0)$ finishes the proof.
\end{proof}

\begin{proof}[Proof of Theorem~\ref{thm:1}]
With in hand our sequence of approximate multi-solitons satisfying the desired estimate, the only thing left to do is to prove that it actually converges to a solution of~\eqref{eq:nlkg2} satisfying the same estimate~\eqref{eq:uniform}. 

First of all, we show the convergence of initial data. Since $U_n$ satisfies~\eqref{eq:uniform}, the sequence $U_n(T_0)$ is bounded in $H^1(\R^d)\times L^2(\R^d)$. Therefore there exists $U_0\in H^1(\R^d)\times L^2(\R^d)$ such that $U_n\rightharpoonup U_0$ weakly in $H^1(\R^d)\times L^2(\R^d)$. We are going  to prove that the previous convergence is strong in  $H^s(\R^d)\times H^{s-1}(\R^d)$ for any $0<s<1$. Take $\varepsilon>0$. Using Lemma~\ref{lem:compact}, we infer the existence of $M_\varepsilon>0$ such that for $n$ large enough
\begin{multline}\label{eq:compact}
\int_{|x|>M_\varepsilon}|\nabla U_{n,1}(T_0)|^2+|U_{n,1}(T_0)|^2+|U_{n,2}(T_0)|^2dx+\\
\int_{|x|>M_\varepsilon}|\nabla U_{0,1}(T_0)|^2+|U_{0,1}(T_0)|^2+|U_{0,2}(T_0)|^2dx\leq \frac\varepsilon2.
\end{multline}
Define $\chi_\varepsilon:\R^d\to[0,1]$ a cutoff function such that $\chi_\varepsilon(x)=1$ if $|x|<M_\varepsilon$, $\chi_\varepsilon(x)=0$ if $|x|>2M_\varepsilon$, and $\norm{\nabla\chi_\varepsilon}_\infty\leq 1$.
We have
\[
\norm{U_n(T_0)-U_0}_{H^s\times H^{s-1}}\leq \norm{(U_n(T_0)-U_0)\chi_\varepsilon}_{H^s\times H^{s-1}}+\norm{(U_n(T_0)-U_0)(1-\chi_\varepsilon)}_{H^s\times H^{s-1}}
\]
From the compactness\footnote{a proof of this fact is included in the Appendix, Lemma~\ref{lem:rellich} for the reader's convenience.} of the injection $H^s(\Omega)\hookrightarrow H^{s-\delta}(\Omega)$ when $\Omega$ is bounded and $\delta>0$ , we infer that, for $n$ large enough and maybe up to a subsequence, we have
\[
\norm{(U_n(T_0)-U_0)\chi_\varepsilon}_{H^s\times H^{s-1}}\leq \frac{\varepsilon}{2}.
\]
Moreover, by~\eqref{eq:compact}
\[
\norm{(U_n(T_0)-U_0)(1-\chi_\varepsilon)}_{H^s\times H^{s-1}}\leq \norm{(U_n(T_0)-U_0)(1-\chi_\varepsilon)}_{H^1\times L^2}\leq \frac{\varepsilon}{2}
\]
 Combining the last three equations gives us
\[
\norm{U_n(T_0)-U_0}_{H^s\times H^{s-1}}\leq \varepsilon.
\]
Hence $U_n(T_0)$ converges strongly to $U_0$ in $H^s(\R^d)\times H^{s-1}(\R^d)$. 

Let us now show that the solution $U$  of~\eqref{eq:nlkg2} in $H^1(\R^d)\times L^2(\R^d)$ with data $U(T_0)=U_0$ satisfies the required estimate. 
By Local Cauchy Theory of~\eqref{eq:nlkg2} in $H^s(\R^d)\times H^{s-1}(\R^d)$, we have the strong convergence 
\[
U_n(t)\to U(t)\qquad \text{ in } H^s(\R^d)\times H^{s-1}(\R^d).
\]
for any $t\in[T_0,+\infty)$. In addition, by uniqueness of the limit and since $U_n(t)$ is bounded in $H^1(\R^d)\times L^2(\R^d)$ (by~\eqref{eq:uniform}), we have the weak convergence for $t\in[T_0,+\infty)$
\[
U_n(t)\rightharpoonup U(t)\qquad \text{ in } H^1(\R^d)\times L^2(\R^d).
\]
Therefore, by weak lower semi-continuity of the $H^1\times L^2$-norm and~\eqref{eq:uniform}, we have
\[
\norm{U(t)-R(t)}_{H^1\times L^2}\leq \liminf_{n\to+\infty}\norm{U_n(t)-R(t)}_{H^1\times L^2}\leq e^{-\alpha\sqrt{m-\omega_\star^2}v_\star t},
\]
which concludes the proof of Theorem~\ref{thm:1}.
\end{proof}

\section{Properties of the profiles}\label{sec:variational}

Since we will be working mainly within the Hamiltonian formulation of~\eqref{eq:nlkg}, it will be convenient to characterize the soliton profiles using the conserved quantities. We already mentioned that the profile $\Phi_{\omega,v}$ is a critical point of the functional \emph{action}
\[
S:=E+\frac{\omega}{\gamma}Q+v\cdot P,
\]
or more explicitly a solution to
\begin{equation}\label{eq:snlkg}
\left\{
\begin{aligned}
-\Delta w_1 +mw_1-|w_1|^{p-1}w_1+i\frac{\omega }{\gamma}w_2-v\cdot\nabla w_2&=0,\\
w_2-i\frac{\omega }{\gamma} w_1+v\cdot\nabla w_1&=0.
\end{aligned}
\right.
\end{equation}
 In this section, we are going to give some variational characterizations of  $\Phi_{\omega,v}$ and study the Hessian $S''(\Phi_{\omega,v})$. 

As far as we know, the variational characterizations given in the following Proposition~\ref{prop:variational} were never derived before, although they are expected in view of what happens in the scalar setting. The ideas on the relationships between different variational characterizations  used further in this section  were introduced by Jeanjean and Tanaka in~\cite{JeTa03-1,JeTa03-2}  (see also~\cite{BeVi10} for related results). We believe that these variational characterizations of the profile $\Phi_{\omega,v}$  are of independent interest. 

For the purpose of constructing multi-solitons, the main result of this section is the coercivity property given in Lemma~\ref{lem:coercivity}. The proof of this result relies on the variational characterization of the profiles as well as on their non-degeneracy, which is given by Lemma~\ref{lem:kernel}.
We shall follow closely the presentation made in~\cite{Le09} for the standing waves of \textsc{nls}.

\subsection{Variational Characterizations}

We define the \emph{mountain pass level} by
\[
MP:=\inf_{\eta\in\Gamma}\sup_{s\in[0,1]}S(\eta(s))
\]
where $\Gamma$ is the set of admissible paths
\[
\Gamma:=\{\eta\in\mathcal C \left((0,1),H^1(\R^d)\times L^2(\R^d)\right);\;\eta(0)=0,S(\eta(1))<0\}.
\]
We define the Nehari constraint for $W\in H^1(\R^d)\times L^2(\R^d)$ by
\[
I(W):=\dual{S'(W)}{W}
\]
and the Nehari level by
\[
NL:=\min\{S(W);\;I(W)=0,W\neq0\}.
\]
We also define the \emph{least energy level} by 
\[
LE:=\min\{S(W);\;W\in H^1(\R^d)\times L^2(\R^d),\,W\not\equiv0,\,S'(W)=0\}.
\]

\begin{proposition}\label{prop:variational}
The profile $\Phi_{\omega,v}$ admits the following variational characterizations:
\[
S(\Phi_{\omega,v})=MP=NL=LE.
\]
\end{proposition}

Let us start by proving using mountain pass arguments that $S$ admits a  critical point. Then we will show that this critical point is at the mountain pass level and also at the least energy level and at the Nehari level and we will identify it with $\Phi_{\omega,v}$. 
\begin{lemma}\label{lem:critical-point}
There exists $\Psi\in H^1(\R^d)\times L^2(\R^d)$ a non-trivial critical point of $S$, i.e. 
\[
\Psi\neq 0,\qquad S'(\Psi)=0.
\]
\end{lemma}

Before going further, we make the following useful observations on the formulation of $S$: for $W=(w_1,w_2)\in H^1(\R^d)\times L^2(\R^d)$ it is simple algebra to see that
\begin{multline*}
S(W)
=\frac{1}{2} \norm{\nabla w_1}_2^2-\frac12\norm{v\cdot{\nabla w_1}}_2^2+\frac{1}{2}\left(m-\frac{\omega^2}{\gamma^2}\right) \norm{w_1}_2^2-\frac12\frac \omega \gamma v\cdot\Im\int_{ \R^d}w_1\nabla \bar w_1\\
+\frac12\norm[\big]{v\cdot\nabla w_1-i\frac{\omega}{\gamma}w_1+w_2}_2^2-\frac{1}{p+1}\norm{w_1}_{p+1}^{p+1}.
\end{multline*}
We can remark further that if $\tilde w_1$ is such that $ w_1(x)=e^{-i\omega \gamma v\cdot x}\tilde w_1(x+(\gamma-1)x_v)$ then we have
\begin{multline}\label{eq:S-reformulated}
S(W)
=\frac1\gamma\left(\frac12 \norm{\nabla \tilde w_1}_2^2+\frac{1}{2}\left(m-\omega^2\right) \norm{\tilde w_1}_2^2\right)
+\frac12\norm[\big]{v\cdot\nabla w_1-i\frac{\omega}{\gamma}w_1+w_2}_2^2-\frac{1}{p+1}\norm{w_1}_{p+1}^{p+1}.
\end{multline}
We shall also use the following Lemma at several occasions.

\begin{lemma}\label{lem:rewriting-trick}
For any $\varepsilon>0$ there exists $\delta>0$ such that for any $W=(w_1,w_2) \in H^1(\R^d)\times L^2(\R^d)$ we have 
\[
\varepsilon\norm{w_1}_{H^1}^2+\norm{v\cdot\nabla w_1-i\frac\omega\gamma w_1+w_2}_2^2\geq \delta\norm{W}_{H^1\times L^2}^2.
\]
\end{lemma}

\begin{proof}
We only have to make $\norm{w_2}_2^2$ appear.
 Write 
\[
w_2=\alpha\left(v\cdot\nabla w_1-i\frac{\omega}{\gamma}w_1\right)+w_2^\perp,
\]
 where $\alpha\in\R$ and $\psld{v\cdot\nabla w_1-i\frac{\omega}{\gamma}w_1}{w_2^\perp}=0$. We have 
\begin{gather}
\norm[\big]{v\cdot\nabla w_1-i\frac{\omega}{\gamma}w_1+w_2}_2^2=
(1+\alpha)^2\norm[\big]{v\cdot\nabla w_1-i\frac{\omega}{\gamma}w_1}_2^2+\norm{w_2^\perp}_2^2\label{eq:possible-degen}\\
\norm{w_2}_2^2=
\alpha^2\norm[\big]{v\cdot\nabla w_1-i\frac{\omega}{\gamma}w_1}_2^2+\norm{w_2^\perp}_2^2\nonumber.
\end{gather}
There is a possible degeneracy in~\eqref{eq:possible-degen} if $\alpha=-1$, but we can compensate it by using a piece of $\norm{w_1}^2_{H^1}$:
\begin{multline}\label{eq:rewriting-trick}
\frac{\varepsilon}{2}\norm{w_1}^2_{H^1}+ \norm[\big]{v\cdot\nabla w_1-i\frac{\omega}{\gamma}w_1+w_2}_2^2\geq \\
 C\left((1+\alpha)^2+\frac{\varepsilon}{2}\right)\norm[\big]{v\cdot\nabla w_1-i\frac{\omega}{\gamma}w_1}_2^2+C\norm{w_2^\perp}_2^2
\geq \tilde C \norm{w_2}_2^2.
\end{multline}
The desired inequality is then a direct consequence of~\eqref{eq:rewriting-trick}.
\end{proof}

\begin{proof}[Proof of Lemma~\ref{lem:critical-point}]
\noindent\emph{Step 1 : Mountain-Pass geometry.}

We claim that the functional $S$ has a mountain-pass geometry, i.e
\[
MP>0.
\]
We start by showing that $\Gamma$ is not empty. Indeed, take $W=(w_1,w_2)\in H^1(\R^d)\times L^2(\R^d)$ and $s>0$. Then using~\eqref{eq:S-reformulated} we see that 
\begin{multline*}
S(sW)
=\frac{s^2}{2} \bigg(
\frac1\gamma\Big(\frac12 \norm{\nabla \tilde w_1}_2^2+\frac{1}{2}\left(m-\omega^2\right) \norm{\tilde w_1}_2^2\Big)
+\frac12\norm[\big]{v\cdot\nabla w_1-i\frac{\omega}{\gamma}w_1+w_2}_2^2\bigg)
-\frac{s^{p+1}}{p+1}\norm{w_1}_{p+1}^{p+1}.
\end{multline*}
Therefore if $s$ is large enough we have $S(sW)<0$, hence the path $s\mapsto S(\frac{s}{C}W)$ belongs to $\Gamma$ provided $C$ has been chosen large enough.

To show that $MP>0$, it is enough to prove that there exists a function $f:\R^+\to\R$ such that $f(s)>0$ for $s$ close to $0$ and  $S(W)\geq f(\norm{W}_{H^1\times L^2})$.
Using~\eqref{eq:S-reformulated}, the continuity of $w_1\to\tilde w_1$ in $H^1(\R^d)$ and Sobolev embeddings, it is easy to see that there exists $\varepsilon>0$ such that
\[
S(W)\geq \frac\varepsilon 2\norm{ w_1}_{H^1}^2+\frac12\norm[\big]{v\cdot\nabla w_1-i\frac{\omega}{\gamma}w_1+w_2}_2^2-C\norm{w_1}_{H^1}^{p+1}.
\]
From Lemma~\ref{lem:rewriting-trick} we infer that there exists $\tilde\delta>0$ such that
\[
S(W)\geq \tilde\delta\norm{W}_{H^1\times L^2}^2-C\norm{W}_{H^1\times L^2}^{p+1}.
\]
This implies that $S(W)>0$ if $\norm{W}_{H^1\times L^2}$ is small and there exist $C>0$, $\delta>0$ such that $S(W)>C>0$ for $\norm{W}_{H^1\times L^2}=\delta$. This implies $MP>0$ and $S$ has a mountain pass geometry.

\noindent\emph{Step 2 : Existence of a Palais-Smale sequence.}

From Ekeland variational principle (see e.g.~\cite{Wi96}) and Step 1, we infer the existence of a Palais-Smale sequence $W_n=(w_{1,n},w_{2,n})$ at the level $MP$, i.e. 
\begin{equation}\label{eq:PS}
S(W_n)\to MP, \quad S'(W_n)\to 0,\qquad \text{as }n\to+\infty.
\end{equation}

\noindent\emph{Step 3 : Non-vanishing of the Palais-Smale sequence.}

Assume by contradiction that the sequence $W_n$ is vanishing, more precisely for any $R>0$ we have 
\[
\lim_{n\to+\infty}\sup_{y\in\R^d}\int_{|x-y|<R}\left(|w_{1,n}|^2+|w_{2,n}|^2\right)dx=0.
\]
Take $\varepsilon>0$ and $R>0$ and let $n$ be large enough so that 
\[	
\sup_{y\in\R^d}\int_{|x-y|<R}\left(|w_{1,n}|^2+|w_{2,n}|^2\right)dx<\varepsilon.
\]
Recall Lions' Lemma (see~\cite{Li84}): for any $w\in H^1(\R^d)$ we have
\begin{equation}\label{eq:lions}
\norm{w}_{p+1}^{p+1}\leq C\left( \sup_{y\in\R^d}\int_{|x-y|<R}|w|^2dx \right)^{p-1}\norm{w}_{H^1}^2.
\end{equation}
Therefore, for $n$ large enough, and using~\eqref{eq:S-reformulated} and Lemma~\ref{lem:rewriting-trick}, we get
\[
S(W_n)\geq C\norm{W_n}_{H^1\times L^2}^2-\varepsilon \norm{w_{1,n}}^2_{H^1},
\]
which implies (if $\varepsilon$ has been chosen small enough) that $(W_n)$ is bounded in $H^1(\R^d)\times L^2(\R^d)$. This boundedness has two consequences: as $n\to+\infty$, we have
\begin{equation}\label{eq:two-consequences}
\dual{S'(W_n)}{W_n}\to0\qquad \text{and}\qquad\norm{w_{1,n}}_{p+1}\to0.
\end{equation}
where the first limit is due to the fact that $(W_n)$ is a Palais-Smale sequence (see~\eqref{eq:PS}) and the second limit comes from~\eqref{eq:lions}. However, we have 
\[
\left(\frac{1}{p+1}-\frac12\right)\norm{w_{1,n}}_{p+1}^{p+1}=S(W_n)-\frac12\dual{S'(W_n)}{W_n}
\]
and therefore~\eqref{eq:two-consequences} implies
\[
\lim_{n\to+\infty}S(W_n)=0,
\]
which enters in contradiction with $\lim_{n\to+\infty}S(W_n)=MP>0$. Therefore the sequence $(W_n)$ is non-vanishing.

\emph{Step 4 : Convergence to a critical point}

 Since $W_n$ is non-vanishing, there exists $R,\delta>0$ and $(y_n)\subset \R^d$ such that for $n$ large enough
\begin{equation}\label{eq:non-vanishing}
\int_{|x-y_n|<R}|w_{1,n}|^2dx>\delta.
\end{equation}
If we substitute  $W_n(\cdot-y_n)$ to $W_n$ (keeping the same notation), the sequence $( W_n)$ is still a Palais-Smale sequence and keeps the same properties. In particular, as known from Step 3,  $( W_n)$ is bounded in $H^1(\R^d)\times L^2(\R^d)$, and therefore we have the existence of $\Psi\in H^1(\R^d)\times L^2(\R^d)$ such that $ W_n\rightharpoonup \Psi$ weakly in $H^1(\R^d)\times L^2(\R^d)$. Since $ W_n$ is a Palais-Smale sequence and $S'$ is continuous, we have $S'(\Psi)=0$. Hence we only have to show that $\Psi$ is non-trivial. This is a direct consequence of~\eqref{eq:non-vanishing} and the compact injection $H^1(|x|<R)\hookrightarrow  L^2(|x|<R)$. Hence $\Psi$ is a non-trivial critical point of $S$ and the proof of Lemma~\ref{lem:critical-point} is finished.
\end{proof}

We turn now to the variational characterizations of the critical point obtained in Lemma~\ref{lem:critical-point}.

\begin{lemma}\label{lem:variational}
Take $\Psi$ the critical point of $S$ found in Lemma~\ref{lem:critical-point}. The following equality is satisfied.
\[
S(\Psi)=MP=NL=LE.
\]
\end{lemma}

\begin{proof}
Let us start by showing
\begin{equation}\label{eq:inequality-1}
S(\Psi)\leq MP.
\end{equation}
Using $S'(\Psi)=0$, we have
\begin{multline}\label{eq:lower-semi}
S(\Psi)=S(\Psi)-\frac{1}{p+1}\dual{S'(\Psi)}{\Psi}=
\left(\frac{1}{2}-\frac{1}{p+1}\right)\bigg(  \frac1\gamma\Big(\norm{\nabla \tilde w_1}_2^2+
\left(m-\omega^2\right) \norm{\tilde w_1}_2^2\Big)
\\+\norm[\big]{v\cdot\nabla w_1-i\frac{\omega}{\gamma}w_1+w_2}_2^2 \bigg)\leq \dots
\end{multline}
Recall that $\tilde w_1$ is such that $ w_1(x)=e^{-i\omega \gamma v\cdot x}\tilde w_1(x+(\gamma-1)x_v)$ (see~\eqref{eq:S-reformulated}).
Using the weak lower semi-continuity of the norm, 
we can continue the inequality started in~\eqref{eq:lower-semi} by
\begin{multline}\label{eq:lower-semi-2}
\dots\leq
\left(\frac{1}{2}-\frac{1}{p+1}\right)\liminf_{n\to+\infty}\bigg( \frac1\gamma \Big(\norm{\nabla  \tilde w_{1,n}}_2^2+
\left(m-\omega^2\right) \norm{\tilde w_{1,n}}_2^2\Big)
\shoveright{+\norm[\big]{v\cdot\nabla  w_{1,n}-i\frac{\omega}{\gamma} w_{1,n}+ w_{2,n}}_2^2 \bigg)}\\
=\liminf_{n\to+\infty}\left( S( W_n)-\frac{1}{p+1}\dual{S'( W_n)}{ W_n}\right).
\end{multline}
Since $( W_n)$ is a Palais-Smale sequence we have 
\[
\lim_{n\to+\infty}\left(S( W_n)-\frac{1}{p+1}\dual{S'( W_n)}{ W_n}\right)=MP,
\]
and we can conclude from~\eqref{eq:lower-semi} and~\eqref{eq:lower-semi-2} that $\Psi$ verifies~\eqref{eq:inequality-1}.

We continue by showing that 
\begin{equation}\label{eq:inequality-2}
MP\leq NL.
\end{equation}
Take an element of the Nehari manifold $W\in H^1(\R^d)\times L^2(\R^d)$, $I(W)=0$. The idea, as in~\cite{JeTa03-1,JeTa03-2}, is to construct a path in $\Gamma$ so that $S(\eta(s))$ achieves its maximum when $\eta(s)=W$. It is easy to see that for $C$ large enough the path $\eta_C$ defined by $\eta_C(s)=CsW$ fulfills our needs. Indeed, we have
\begin{multline*}
\frac{ \partial }{\partial s}S(sW)=s\bigg( \frac1\gamma\Big( \norm{\nabla \tilde w_1}_2^2+\left(m-\omega^2\right) \norm{\tilde w_1}_2^2\Big)
+\norm[\big]{v\cdot\nabla w_1-i\frac{\omega}{\gamma}w_1+w_2}_2^2 
-s^{p-1}\norm{w_1}_{p+1}^{p+1}
\bigg).
\end{multline*}
In particular, $\frac{ \partial }{\partial s}S(sW)|_{s=1}=I(W)=0$. Therefore $\frac{ \partial }{\partial s}S(sW)>0$ for $s\in(0,1)$ and $\frac{ \partial }{\partial s}S(sW)<0$ for $s>1$. Hence the path $S(\eta_C(s))$ achieves its maximum when $s=\frac1C$ and $\eta(\frac1C)=W$. Therefore,
\[
ML\leq S(W),
\]
and since this is true for any $W$ on the Nehari manifold this proves~\eqref{eq:inequality-2}.

It is easy to see that
\begin{equation}\label{eq:inequality-3}
NL\leq LE.
\end{equation}
Indeed, any solution $W$ of~\eqref{eq:snlkg} (i.e. any critical point of $S$) satisfies the Nehari identity $I(W)=0$. Thus the infimum for $NL$ is taken on a larger set than the infimum for $LE$, hence~\eqref{eq:inequality-3}.

Finally, as a direct consequence of $S'(\Psi)=0$ and the definition of $LE$ we have
\begin{equation}\label{eq:inequality-4}
LE\leq S(\Psi).
\end{equation}
Combining~\eqref{eq:inequality-1},~\eqref{eq:inequality-2},~\eqref{eq:inequality-3},~\eqref{eq:inequality-4} finishes the proof of Lemma~\ref{lem:variational}.
\end{proof}

\begin{proof}[Proof of Proposition~\ref{prop:variational}]
In view of Lemmas~\ref{lem:critical-point} and~\ref{lem:variational}, the only thing left to prove is that $\Psi=\Phi_{\omega,v}$. Let us first see the case $v=0$. Since $\Psi=(\psi_1,\psi_2)$ is a critical point of $S$, we have $\psi_2=i\omega\psi_1$. Therefore, since $\varphi_\omega$ is a ground state of~\eqref{eq:scalar-snlkg}, we have
\begin{multline*}
S(\Psi)=
\frac{1}{2} \norm{\nabla \psi_1}_2^2+\frac{1}{2}\left(m-\omega^2\right) \norm{\psi_1}_2^2-\frac{1}{p+1}\norm{\psi_1}_{p+1}^{p+1}
\\
\geq
\frac{1}{2} \norm{\nabla \varphi_\omega}_2^2+\frac{1}{2}\left(m-\omega^2\right) \norm{\varphi_\omega}_2^2-\frac{1}{p+1}\norm{\varphi_\omega}_{p+1}^{p+1}=S(\Phi_{\omega,0}).
\end{multline*}
Therefore when $v=0$, we indeed have $\Psi=\Phi_{\omega,0}$. Let us now treat the case $v\neq 0$. Let $\tilde\psi_1$ be such that $\psi_1(x)=e^{-i\omega\gamma v\cdot x}\tilde\psi_1(x+(\gamma-1)x_v)$ and define $\tilde\psi_2:=i\omega\tilde\psi_1$. Then $\tilde\Psi:=(\tilde\psi_1,\tilde\psi_2)$ is a solution to~\eqref{eq:snlkg} with $v=0$. Indeed, it is not hard to see that
\[
-\Delta \psi_1+i\frac\omega\gamma \psi_2-v\cdot\nabla\psi_2=e^{-i\omega\gamma v\cdot x}\left(-\Delta\tilde\psi_1-\omega^2\tilde\psi_1\right).
\]
Hence 
\[
 S(\Psi)=\frac1\gamma(E+\omega Q)(\tilde\Psi)\geq\frac1\gamma (E+\omega Q)(\Phi_{\omega,0})=S(\Phi_{\omega,v}).
\]
This implies that $\Psi=\Phi_{\omega,v}$ for any $v$ and finishes the proof of Proposition~\ref{prop:variational}.
\end{proof}

\subsection{Kernel}

\begin{lemma}\label{lem:kernel}
The following description holds for the kernel of $S''(\Phi_{\omega,v})$:
\[
\mathrm{Ker}(S''(\Phi_{\omega,v}))=\Span\{ i\Phi_{\omega,v} , \nabla  \Phi_{\omega,v}\}.
\]
\end{lemma}

\begin{proof}
The inclusion $\supset$ is easy to obtain. Indeed, due to invariance by translation and phase shifts,  for any $\theta\in\R$ and $y\in\R^d$ we have
\[
S'(e^{i\theta}\Phi_{\omega,v}(\cdot+y))=0.
\]
The result is obtained by deriving with respect to $\theta$ and $y$ at $\theta=0,y=0$. The reverse inclusion is much more delicate. We shall rely on existing results for standing waves of \textsc{nls} to prove it. First remark that if $W=(w_1,w_2)$ belongs to the kernel of $S''(\Phi_{\omega,v})$, then it satisfies 
\[
\left\{
\begin{aligned}
-\Delta w_1+m^2w_1-(p-1)|\Phi_{\omega,v}^1|^{p-3}\Phi_{\omega,v}^1\Re(\Phi_{\omega,v}^1\bar w_1)\qquad\\+|\Phi_{\omega,v}^1|^{p-1}w_1+i\frac{\omega}{\gamma}w_2-v\cdot\nabla w_2=0,\\
w_2-i\frac{\omega}{\gamma}w_1+v\cdot\nabla w_1=0.
\end{aligned}
\right.
\]
Here, we have denoted by $\Phi^1_{\omega,v}$ the first component of $\Phi_{\omega,v}$, i.e.
\[
\Phi_{\omega,v}^1:=e^{-i\gamma\omega v\cdot x}\varphi_{\omega}(x+(\gamma-1)x_{v}).
\]
Take $\tilde W:=(\tilde w_1,\tilde w_2)$ such that
\begin{equation*}
\begin{pmatrix} w_1\\ w_2\end{pmatrix}=e^{-i\gamma \omega v\cdot x}\begin{pmatrix}\tilde w_1(x+(\gamma-1)x_v)\\\gamma\tilde w_2(x+(\gamma-1)x_v)-\gamma v\nabla \tilde w_1(x+(\gamma-1)x_v)\end{pmatrix}
\end{equation*}
It is a lengthy but straightforward computation to verify that $\tilde W$ satisfies
\[
\left\{
\begin{aligned}
-\Delta \tilde w_1+m^2\tilde w_1-(p-1)|\Phi_{\omega,0}^1|^{p-3}\Phi_{\omega,0}^1\Re(\Phi_{\omega,0}^1\overline {\tilde w_1})\qquad\\-|\Phi_{\omega,0}^1|^{p-1}\tilde w_1+i{\omega}\tilde w_2=0,
\\
\tilde w_2-i{\omega}\tilde w_1=0.
\end{aligned}
\right.
\]
Remembering now that $\Phi_{\omega,0}^1=\varphi_\omega$ and using the second equation to substitute in the first we get
\begin{equation}\label{eq:known-ground}
\left\{
\begin{aligned}
-\Delta \tilde w_1+(m-\omega^2)\tilde w_1-(p-1)|\varphi_\omega|^{p-1}\Re( \tilde w_1)-|\varphi_\omega|^{p-1}\tilde w_1&=0,
\\
\tilde w_2&=i{\omega}\tilde w_1.
\end{aligned}
\right.
\end{equation}
Fortunately we arrive on a known ground: it is well-known since the celebrated work of Weinstein~\cite{We85} and Kwong~\cite{Kw89} (see~\cite{ChGuNaTs07} for a modern short proof of this result) that the only solutions to~\eqref{eq:known-ground} are 
\[
\binom{\tilde w_1}{\tilde w_2}\in\Span\left\{\binom{\nabla \varphi_\omega}{i\omega\nabla\varphi_\omega};\binom{i \varphi_\omega}{-\omega\varphi_\omega}\right\}=\Span\left\{\nabla\Phi_{\omega,0};i\Phi_{\omega,0}\right\}.
\]
Coming back into the original variables, this implies that 
\[
W\in\Span\left\{\nabla\Phi_{\omega,v};i\Phi_{\omega,v}\right\}
\]
and finishes the proof.
\end{proof}

\subsection{Coercivity}

The proof of our result relies on the fact that the solitary waves we are considering are stable. In particular, we have at our disposal a coercivity property on the Hessian of the action $S$ related to the soliton profile $\Phi_{\omega,v}$ which allows us to control the difference between a soliton and a function in a neighborhood of its orbit. The coercivity property is the following.

\begin{lemma}[Coercivity]\label{lem:coercivity}
Assume $p<1+\frac4d$ and let $\omega\in\R$ and $v\in\R^d$ be such that
$
\frac{1}{1+\frac{4}{p-1}-d}<\frac{\omega^2}{m}<1,$ and $|v|<1$ (i.e. they are compatible with $\mathcal O_{\rm stab}$ defined in~\eqref{eq:O-stab}) and let $\Phi_{\omega,v}$ be the associated ground state.
There exists $\delta$ such that for any $W\in H^1(\R^d)\times L^2(\R^d)$ satisfying the orthogonality conditions 
\begin{equation}\label{eq:moving-ortho}
\psld{ W}{\nabla\Phi_{\omega,v}}=\psld{W}{iJ\Phi_{\omega,v}}=\psld{W}{i\Phi_{\omega,v}}=0
\end{equation}
we have
\[
H_{\omega,v}(W)\geq \delta\norm*{W}^2_{H^1\times L^2}.
\]
where for brevity in notation 
 we defined
\[
H(W):=\dual{S''(\Phi_{\omega,v})W}{W}.
\]
\end{lemma}

Similar results date back to the work  of  Weinstein~\cite{We85,We86}  for \textsc{nls} equations. These ideas were later  generalized by Grillakis, Shatah and Strauss~\cite{GrShSt87} in a abstract setting. More recently, Stuart~\cite{St01} described precisely the orbital stability of solitons of NLKG using also  a coercivity statement, but with different orthogonality conditions and a slightly more complicated proof.

\begin{proof}[Proof of Lemma~\ref{lem:coercivity}]
\emph{Step 1 : Analysis of the spectrum of $S''(\Phi_{\omega,v})$.}

We first remark that, due the exponential localization of $\Phi_{\omega,v}$, the operator $S''(\Phi_{\omega,v})$ is a compact perturbation of the self-adjoint operator
\[
\mathbb L:=
\begin{pmatrix}
-\Delta+m&0\\
0&1
\end{pmatrix}
+
\frac\omega\gamma\begin{pmatrix}
0&i\\
-i&0
\end{pmatrix}
+
v\cdot
\begin{pmatrix}
0&-\nabla\\
\nabla&0
\end{pmatrix},\quad D(\mathbb L)=H^2(\R^d)\times H^1(\R^d).
\]
By Weyl's Theorem, $S''(\Phi_{\omega,v})$ and $\mathbb L$ share the same essential spectrum, that we now analyze.
Observe that for $W=(w_1,w_2)\in H^2(\R^d)\times H^1(\R^d)$ we have
\[
\dual{\mathbb L W}{W}=\norm{\nabla w_1}_2^2+m\norm{w_1}_2^2+\norm{w_2}_2^2-2\frac\omega\gamma\psld{i w_1}{w_2}+2v\cdot\psld{\nabla w_1}{w_2},
\]
which, after some factorizations (similar to those used in~\eqref{eq:S-reformulated}), we can rewrite 
\[
\dual{\mathbb L W}{W}=
\frac1\gamma\left( \norm{\nabla \tilde w_1}_2^2+\left(m-\omega^2\right) \norm{\tilde w_1}_2^2\right)
+\norm[\big]{v\cdot\nabla w_1-i\frac{\omega}{\gamma}w_1+w_2}_2^2
\]
where $\tilde w_1$ is such that $ w_1(x)=e^{-i\omega \gamma v\cdot x}\tilde w_1(x+(\gamma-1)x_v)$. 
From Lemma~\ref{lem:rewriting-trick}, we see that there exists $\delta>0$ such that for any $W\in H^2(\R^d)\times H^1(\R^d)$ we have 
\[
\dual{\mathbb L W}{W}\geq \delta\norm{W}_{H^1\times L^2}^2.
\]
This implies that the essential spectrum of $S''(\Phi_{\omega,v})$ is positive and away from $0$. The rest of its spectrum consists in a finite number of isolated eigenvalues. It turns out that from the variational characterization of $\Phi_{\omega,v}$, we can infer that $S''(\Phi_{\omega,v})$ has Morse Index 1, i.e. it admits only one negative simple eigenvalue (see e.g.~\cite{AmMa07}). We denote this eigenvalue by $-\lambda<0$, and $\Psi$ an associated normalized eigenvector, i.e. $S''(\Phi_{\omega,v})\Psi=-\lambda\Psi$ and $\norm{\Psi}_{L^2\times L^2}=1$. 

\emph{Step 2 : A positivity property.
We prove now that if $W\in H^1(\R^d)\times L^2(\R^d)$ satisfies the orthogonality conditions~\eqref{eq:moving-ortho}, then 
\[
\dual{S''(\Phi_{\omega,v})W}{W}>0.
\]
}
A particular vector associated with $S''(\Phi_{\omega,v})$ is $\Lambda_{\omega}\Phi_{\omega,v}$, where $\Lambda_{\omega}:=\frac{\partial}{\partial \omega}$. Indeed, deriving $S'(\Phi_{\omega,v})=0$ with respect to $\omega$ we get 
\begin{equation}\label{eq:miracle}
S''(\Phi_{\omega,v})\Lambda_{\omega}\Phi_{\omega,v}=-\frac1\gamma Q'(\Phi_{\omega,v})=-\frac1\gamma iJ\Phi_{\omega,v}.
\end{equation}
This implies, using~\eqref{eq:scaled-phi},
\begin{multline}\label{eq:negativity-1}
\dual{S''(\Phi_{\omega,v})\Lambda_{\omega}\Phi_{\omega,v}}{\Lambda_{\omega}\Phi_{\omega,v}}=-\frac1\gamma\dual{Q'(\Phi_{\omega,v})}{\Lambda_{\omega}\Phi_{\omega,v}}=-\frac1\gamma\Lambda_{\omega}Q(\Phi_{\omega,v})
\\
=\Lambda_\omega\left(\frac{\omega}{\gamma}\norm{\varphi_\omega}_2^2\right)=
\frac{(m-\omega^2)^{\frac{2}{p-1}-\frac d2}}{\gamma}\left(-\frac{2\omega^2\left(\frac{2}{p-1}+\frac d2\right)}{m-\omega^2}-1\right)\norm{\tilde \varphi}_2^2
<0
\end{multline}
where the last inequality follows from the fact that $\omega$ is compatible with $\mathcal O_{\mathrm{stab}}$ (see the definition~\eqref{eq:O-stab}). It is easy to verify that  $\Lambda_{\omega}\Phi_{\omega,v}$ is orthogonal to the kernel of $S''(\Phi_{\omega,v})$, namely that we have
\[
\psld{\Lambda_{\omega}\Phi_{\omega,v}}{i\Phi_{\omega,v}}=\psld{\Lambda_{\omega}\Phi_{\omega,v}}{\nabla\Phi_{\omega,v}}=0.
\]
Let us write the orthogonal decomposition of $\Lambda_{\omega}\Phi_{\omega,v}$ along the spectrum of $S''(\Phi_{\omega,v})$:
\begin{equation}\label{eq:spectral-dec-1}
\Lambda_{\omega}\Phi_{\omega,v}=\alpha \Psi+\Pi,
\end{equation}
where $\alpha\neq0$ and $\Pi$ is in the positive eigenspace of $S''(\Phi_{\omega,v})$, in particular
\[
\dual{S''(\Phi_{\omega,v})\Pi}{\Pi}\geq \delta\norm{\Pi}_{H^1\times L^2}^2.
\]
From~\eqref{eq:negativity-1} and~\eqref{eq:spectral-dec-1}, we infer that 
\begin{equation}\label{eq:negativity-2}
-\lambda\alpha^2+\dual{S''(\Phi_{\omega,v})\Pi}{\Pi}=\dual{S''(\Phi_{\omega,v})\Lambda_{\omega}\Phi_{\omega,v}}{\Lambda_{\omega}\Phi_{\omega,v}}<0
\end{equation}

Take now $W\in H^1(\R^d)\times L^2(\R^d)$ satisfying the orthogonality conditions~\eqref{eq:moving-ortho}. We also write the orthogonal decomposition of $W$ along the spectrum of $S''(\Phi_{\omega,v})$:
\begin{equation}\label{eq:spectral-dec-2}
W=\beta \Psi+\Xi,
\end{equation}
where $\beta\in\R$ and $\Xi$ is in the positive eigenspace of $S''(\Phi_{\omega,v})$. If $\beta=0$, the conclusion follows, so we assume $\beta\neq0$. Using~\eqref{eq:moving-ortho},~\eqref{eq:miracle},~\eqref{eq:spectral-dec-1} and~\eqref{eq:spectral-dec-2}, we have
\[
0=\dual{S''(\Phi_{\omega,v})\Lambda_{\omega}\Phi_{\omega,v}}{W}=-\lambda\alpha\beta+ \dual{S''(\Phi_{\omega,v})\Pi}{\Xi}
\]
Note that on the positive spectral subspace of $S''(\Phi_{\omega,v})$, Cauchy-Schwartz inequality holds:
\[
\dual{S''(\Phi_{\omega,v})\Pi}{\Xi}^2\leq \dual{S''(\Phi_{\omega,v})\Pi}{\Pi}\dual{S''(\Phi_{\omega,v})\Xi}{\Xi}.
\]
Therefore, 
\begin{multline*}
\dual{S''(\Phi_{\omega,v})W}{W}=-\lambda\beta^2+\dual{S''(\Phi_{\omega,v})\Xi}{\Xi}\geq
-\lambda\beta^2+ \frac{\dual{S''(\Phi_{\omega,v})\Pi}{\Xi}^2}{\dual{S''(\Phi_{\omega,v})\Pi}{\Pi}}
> -\lambda\beta^2+\frac{(-\lambda\alpha\beta)^2}{\lambda\alpha^2}=0
\end{multline*}

\emph{Step 3. The coercivity property.}

Assume by contradiction that there exists $(W_n=(w^n_1,w_2^n))\subset H^1(\R^d)\times L^2(\R^d)$ satisfying the orthogonality conditions~\eqref{eq:moving-ortho} and such that 
\[
\norm{W_n}_{H^1\times L^2}=1\qquad\text{and}\qquad\lim_{n\to+\infty}\dual{S''(\Phi_{\omega,v})W_n}{W_n}=0.
\]
Recall that, as for~\eqref{eq:S-reformulated}, we have 
\begin{multline*}
\dual{S''(\Phi_{\omega,v})W_n}{W_n}=\frac1\gamma\left( \norm{\nabla \tilde w^n_1}_2^2+\left(m-\omega^2\right) \norm{\tilde w^n_1}_2^2\right)
+\norm[\big]{v\cdot\nabla w_1-i\frac{\omega}{\gamma}w_1+w_2}_2^2
\\
-\int_{\R^d}\left(
(p-1)|\Phi_{\omega,v}^1|^{p-3}\Re(\Phi_{\omega,v}^1\bar w^n_1)^2+|\Phi_{\omega,v}^1|^{p-1}|w^n_1|^2\right)dx.
\end{multline*}
where $\tilde w_1^n$ is such that $ w_1^n(x)=e^{-i\omega \gamma v\cdot x}\tilde w_1^n(x+(\gamma-1)x_v)$. 

Since $(W_n)$ is bounded in $H^1(\R^d)\times L^2(\R^d)$, there exists $W\in H^1(\R^d)\times L^2(\R^d)$ such that 
\[
W_n\rightharpoonup W\;\text{as }n\to+\infty\;\text{ weakly in } H^1(\R^d)\times L^2(\R^d).
\]
On one hand $W$ must  satisfy~\eqref{eq:moving-ortho} and from Step 2 we have, if $W\neq 0$. 
\[
\dual{S''(\Phi_{\omega,v})W}{W}>0.
\]
On the other hand, by weak convergence and exponential decay of $\Phi_{\omega,v}$ we have
\[
\dual{S''(\Phi_{\omega,v})W}{W}\leq \liminf_{n\to+\infty}\dual{S''(\Phi_{\omega,v})W_n}{W_n}=0.
\]
Therefore $W$ must be $W\equiv 0$. However, in this case it would implies 
\[
-\int_{\R^d}\left(
(p-1)|\Phi_{\omega,v}^1|^{p-3}\Re(\Phi_{\omega,v}^1\bar w^n_1)^2+|\Phi_{\omega,v}^1|^{p-1}|w^n_1|^2\right)dx\to0
\]
and since $\norm{W_n}_{H^1\times L^2}=1$, we would have (using Lemma~\ref{lem:rewriting-trick})
\[
\dual{S''(\Phi_{\omega,v})W_n}{W_n}\geq \delta>0,
\]
which is a contradiction.
\end{proof}

\section{Modulation Theory}\label{sec:modulation}

The use of a coercivity property similar to Lemma~\ref{lem:coercivity} but adapted to a multi-soliton (see Lemma~\ref{lem:multi-coercivity}) will require to deal with orthogonality conditions. These  orthogonality conditions will obtained by modulation.

Given parameters $(\varepsilon,L)$, consider a neighborhood of the sum of solitons
\[
\mathcal U(\varepsilon,L):=\Bigg\{
U\in H^1(\R^d)\times L^2(\R^d);\inf_{\substack{\xi_j>\xi_{j-1}+L\\\vartheta_j\in\R\\j=1,\dots,N}} \norm[\bigg]{U-\sum_{j=1}^N  e^{i\vartheta_j}\Phi_j(\cdot-\xi_j) }_{H^1\times L^2}<\varepsilon
\Bigg\}
\]

The main result of this section is the following.

\begin{proposition}[Dynamical Modulation]\label{prop:dyn-mod}
There exists $\tilde\varepsilon,\tilde L,C,\tilde C>0$ such that for any $0<\varepsilon<\tilde\varepsilon$ and $L>\tilde L$ the following property is verified. \\
Let $U(t,x)=(u_1,u_2)(t,x)$ be a solution of~\eqref{eq:nlkg2} satisfying on a time interval $I$ 
\[
U\in\mathcal{U}(\varepsilon,L),\qquad\text{ for all }t\in I.
\]
 For $j=1,\dots,N$, there exist (unique) $\mathcal C^1$ functions
\[ 
\tilde \theta_j:I\to \R, \qquad\tilde\omega_j:I\to (-\sqrt{m},\sqrt{m}),\qquad \tilde x_j:I\to \R^d,
\]
 such that if we define $\tilde R_j(t)$ and $\Upsilon(t)$ by
\begin{equation}\label{eq:def-tilde-Phi}
\tilde R_j(t)=e^{i\tilde\theta_j(t)}\Phi_{\tilde\omega_j(t),v_j}(\cdot-\tilde x_j(t)),\qquad
\Upsilon(t)=U(t)- \sum_{j=1}^N \tilde R_j(t),
\end{equation}
 then $\Upsilon$ satisfies for all $t\in I$ the orthogonality conditions 
\begin{equation}\label{eq:ortho-dyn}
\psld{\Upsilon}{i \tilde R_j}=\psld{\Upsilon}{iJ\tilde R_j}=\psld{\Upsilon}{\nabla\tilde R_j}=0,\quad j=1,\dots,N.
\end{equation}
Moreover,  for all $t\in I$ we have
\[
\norm{\Upsilon}_{H^1\times L^2}+\sum_{j=1}^N|\tilde \omega_j-\omega_j|\leq \tilde C\varepsilon,\qquad \tilde x_{j+1}-\tilde x_j>\frac{L}{2},\; j=1,\dots,N-1,
\]
and the derivatives in time verify
\begin{equation}\label{eq:parameters-derivatives-estimate}
\sum_{j=1}^ N\left(|\partial_t\tilde\omega_j|+\left|\partial_t\tilde\theta_j-\frac{\tilde\omega_j}{\gamma_j}\right|^2+|\partial_t \tilde x_j-v_j|^2\right)<C\left(\norm{\Upsilon}_2^2+e^{-3\alpha\sqrt{m-\omega^2_\star}v_\star t}\right).
\end{equation}
\end{proposition}

The proof of Proposition~\ref{prop:dyn-mod} relies on the following Lemma. Note that this lemma is valid for time-independent functions.

\begin{lemma}[Static Modulation]\label{lem:modulation-static-2}
There exist $\tilde L,\tilde C,\tilde \varepsilon>0$ such that for any  $L>\tilde L$, $0<\varepsilon<\tilde\varepsilon$, the following property is verified. \\
 For $j=1,\dots,N$, there exist (unique) $\mathcal C^1$ functions
\[ 
\tilde \theta_j:\mathcal U(\varepsilon,L)\to \R, \qquad\tilde\omega_j:\mathcal U(\varepsilon,L)\to (-\sqrt{m},\sqrt{m}),\qquad \tilde x_j:\mathcal U(\varepsilon,L)\to \R^d,
\]
 such that if we define $\tilde\Phi_j$ and $\Upsilon$ by
\begin{equation*}
\tilde\Phi_j=e^{i\tilde\theta_j}\Phi_{\tilde\omega_j,v_j}(\cdot-\tilde x_j),\qquad
\Upsilon=U- \sum_{j=1}^N \tilde\Phi_j
\end{equation*}
 then $\Upsilon$ satisfies the orthogonality conditions 
\begin{equation}\label{eq:ortho-stat}
\psld{\Upsilon}{i\tilde\Phi_j}=\psld{\Upsilon}{iJ\tilde\Phi_j}=\psld{\Upsilon}{\nabla\tilde\Phi_j}=0,\quad j=1,\dots,N
\end{equation}
Moreover, 
\[
\norm{\Upsilon}_{H^1\times L^2}+\sum_{j=1}^N|\tilde \omega_j-\omega_j|\leq \tilde C\varepsilon,\qquad \tilde x_{j+1}-\tilde x_j>\frac{L}{2},\; j=1,\dots,N-1.
\]
\end{lemma}

In the proofs, we will use the notation $\Lambda_{\omega_j}$ for the scaling operator, i.e. 
\[
\displaystyle\Lambda_{\omega_j}\tilde\Phi_j:=\frac{\partial}{\partial \omega}e^{i\tilde\theta_j}\Phi_{\omega,v_j}(\cdot-\tilde x_j)\Big|_{\omega=\tilde\omega_j}.
\]

\begin{proof}
We start by proving the lemma in a ball. Take $\varepsilon>0$, $L>0$, $(\vartheta_j)_{j=1,\dots,N}\subset\R$ and $(\xi_j)_{j=1,\dots,N}\subset\R^d$ such that $\xi_{j+1}>\xi_j+L$ for $j=1,\dots,N-1$. Let $\mathcal{B}(\varepsilon)$ denote the ball of $H^1(\R^d)\times L^2(\R^d)$ defined by 
\[
\mathcal{B}(\varepsilon)=
\left\{
U\in H^1(\R^d)\times L^2(\R^d); \norm*{U-\sum_{j=1}^N  e^{i\vartheta_j}\Phi_j(x-\xi_j)}_{H^1\times L^2}<\varepsilon
\right\}
\]
Define
\(
\mathfrak p_0:=(\vartheta_1,\dots,\vartheta_N,\omega_1,\dots,\omega_N,\xi_1,\dots,\xi_N),
\)
and let $\mathfrak P\subset\R^N\times\R^N\times(\R^d)^N$, be a neighborhood of $\mathfrak p_0$. We denote by
\[
\mathfrak p=(\theta_1,\dots,\theta_N,\varpi_1,\dots,\varpi_N,y_1,\dots,y_N)
\]
a generic element of $\mathfrak P$.
We define the functional $F:\mathfrak P\times \mathcal B(\varepsilon)\to(\R^{(d+2)})^N$ by 
\begin{equation*}
F(\mathfrak p,U):=
\begin{pmatrix}
F_{k,1}(\mathfrak p,U)\\
F_{k,2}(\mathfrak p,U)\\
F_{k,3}(\mathfrak p,U)
\end{pmatrix}_{k=1,\dots,N},
\end{equation*}
where for $k=1,\dots,N$ we have set
\begin{align*}
F_{k,1}(\mathfrak p,U)&=\psld{U-\sum_{j=1}^Ne^{i\theta_j}\tau_{y_j}\Phi_{\varpi_j,v_j}}
	{ie^{i\theta_k}\tau_{y_k}\Phi_{\varpi_k,v_k}},\\
F_{k,2}(\mathfrak p,U)&=\psld{U-\sum_{j=1}^Ne^{i\theta_j}\tau_{y_j}\Phi_{\varpi_j,v_j}}		 
	{ie^{i\theta_k}\tau_{y_k}J\Phi_{\varpi_k,v_k}},\\
F_{k,3}(\mathfrak p,U)&=\psld{U-\sum_{j=1}^Ne^{i\theta_j}\tau_{y_j}\Phi_{\varpi_j,v_j}}
	{ e^{i\theta_k}\tau_{y_k}\nabla\Phi_{\varpi_k,v_k}}.
\end{align*}
Here, $\tau_y$ is the translation by $y$, i.e. $\tau_y v(x)=v(x-y)$.
We clearly have 
\[
F\left(\mathfrak p_0,\sum_{j=1}^N  e^{i\vartheta_j}\Phi_j(x-\xi_j)\right)=0.
\]
 The lemma inside the ball will follow from the Implicit Function Theorem if we prove that 
\begin{equation}\label{eq:invertibility}
\frac{\partial F}{\partial\mathfrak p}\bigg(\mathfrak p=\mathfrak p_0, U=\sum_{j=1}^N  e^{i\vartheta_j}\Phi_j(x-\xi_j)\bigg)
\text{ is invertible}.
\end{equation}
The computation of the derivative is not very hard. Many terms will be made small using the exponential decay of the profiles. Other will cancel due to orthogonality. We will essentially be left with a diagonal matrix with nonzero entries, hence the invertibility. We give only some representative calculations. Let's start by
\[
\frac{\partial F_{k,1}}{\partial \theta_j}\bigg(\mathfrak p_0, \sum_{j=1}^N  e^{i\vartheta_j}\Phi_j(x-\xi_j)\bigg)=-\psld{ie^{i\vartheta_j}\tau_{\xi_j}\Phi_{\omega_j,v_j}}{ie^{i\vartheta_k}\tau_{\xi_k}\Phi_{\omega_k,v_k}}
\]
When $j=k$, we readily have 
\[
\frac{\partial F_{k,1}}{\partial \theta_k}\bigg(\mathfrak p_0, \sum_{j=1}^N  e^{i\vartheta_j}\Phi_j(x-\xi_j)\bigg)=-\norm{\Phi_{\omega_k,v_k}}_2^2.
\]
Assume now $j\neq k$. Then, by exponential decay (see~\eqref{eq:exponential-decay}), we have 
\begin{multline*}
\left|\tau_{\xi_j}\Phi_{\omega_j,v_j}\tau_{\xi_k}\Phi_{\omega_k,v_k}\right|
	\leq C e^{-\frac{\sqrt{m-\omega_j^2}}{2}|x-\xi_j|}e^{-\frac{\sqrt{m-\omega_k^2}}{2}|x-\xi_k|}\\
	\leq C e^{-\frac{\sqrt{m-\omega_j^2}}{4}|x-\xi_j|}e^{-\frac{\sqrt{m-\omega_k^2}}{4}|x-\xi_k|}
e^{-\frac{\sqrt{\min\{{m-\omega_j}^2,m-{\omega_k}^2\}} }{4}|\xi_j-\xi_k|}.
\end{multline*}
Therefore, since $|\xi_j-\xi_k|>L$, this implies
\begin{equation}\label{eq:exp-small}
\left|\frac{\partial F_{k,1}}{\partial \theta_j}\bigg(\mathfrak p_0, \sum_{j=1}^N  e^{i\vartheta_j}\Phi_j(x-\xi_j)\bigg)\right|\leq C e^{-\frac{\sqrt{\min\{{m-\omega_j^2},{m-\omega_k^2}\}} }{4}L}.
\end{equation}
This quantity can be made as small as we need by increasing the value of~$L$.
For the derivative with respect to $\varpi_j$, we have 
\[
\frac{\partial F_{k,1}}{\partial \varpi_j}\bigg(\mathfrak p_0, \sum_{j=1}^N  e^{i\vartheta_j}\Phi_j(x-\xi_j)\bigg)=-\psld{e^{i\theta_j}\tau_{\xi_j}\Lambda_{\omega_j}\Phi_{\omega_j,v_j}}{ie^{i\theta_k}\tau_{\xi_k}\Phi_{\omega_k,v_k}}
\]
When $j\neq k$, this quantity can be made small as in~\eqref{eq:exp-small}. For $j=k$, since $\varphi_\omega\in\R$, we simply have 
\[
\frac{\partial F_{k,1}}{\partial \varpi_k}\bigg(\mathfrak p_0, \sum_{j=1}^N  e^{i\vartheta_j}\Phi_j(x-\xi_j)\bigg)=-\frac1\gamma\psld{\Lambda_{\omega_j}\Phi_{\omega_k}}{i\Phi_{\omega_k}}=0.
\]
All other computations follow from similar arguments and we finally find that \[
\frac{\partial F}{\partial\mathfrak p}\bigg(\mathfrak p_0, \sum_{j=1}^N  e^{i\vartheta_j}\Phi_j(x-\xi_j)\bigg)=DIAG+O(e^{-\frac{\sqrt{m-\omega_\star^2}}{4}L})
\]
where $DIAG$ is a diagonal matrix with nonzero entries on the diagonal. Therefore, for $L$ large enough, we have the desired invertibility property~\eqref{eq:invertibility} and the Implicit Function Theorem implies the result inside the ball. 
Since any $U\in\mathcal{U}(\varepsilon,L)$ belongs to some ball $\mathcal{B}(\varepsilon)$, the existence part follows in the cylinder $\mathcal{U}(\varepsilon,L)$. To show uniqueness, one has to prove that the functions obtained are independent of the ball chosen, we leave the details of this argument to the reader.
\end{proof}

\begin{proof}[Proof of Proposition~\ref{prop:dyn-mod}]
The first part of the statement follows from Lemma~\ref{lem:modulation-static-2} (except the regularity that follows from other regularization arguments, see~\cite{MaMe01}), hence the main thing to check is~\eqref{eq:parameters-derivatives-estimate}. We first write the equation verified by $\Upsilon$. Recall that $U$ satisfies
\(
\partial_t U=JE'(U).
\)
We replace $U$ by $\sum_{j=1}^N \tilde R_j(t)+\Upsilon(t)$ in the previous equation to get
\[
\partial_t\Upsilon+\sum_{j=1}^N \left(i \partial_t\tilde\theta_j\tilde R_{j}+\partial_t\tilde\omega_j\Lambda_{\tilde\omega_j}  \tilde R_{j}
-\partial_t\tilde x_j\cdot\nabla\tilde R_{j}\right)=JE'\bigg(\sum_{j=1}^N \tilde R_j(t)+\Upsilon(t)\bigg)
\]
such that it follows
\begin{multline}\label{eq:linearized-eq}
\partial_t\Upsilon+\sum_{j=1}^N\left(i\left(\partial_t\tilde\theta_j-\frac{\tilde\omega_j}{\gamma_j}\right)\tilde R_{j}+\partial_t\tilde\omega_j\Lambda_{\tilde\omega_j}  \tilde R_{j}
-(\partial_t\tilde x_j-v_j)\cdot\nabla\tilde R_{j}\right)
\\=\mathfrak{L}\Upsilon+\mathfrak N(\Upsilon) +O(e^{-3\alpha\sqrt{m-\omega_\star^2}v_\star t})
\end{multline}
where  $\mathfrak L$ is the linearized operator defined by
\begin{gather*}
\mathfrak L:=
J\left(
\begin{pmatrix}
-\Delta+m&0\\
0&1
\end{pmatrix}
-
\sum_{j=1}^N
\begin{pmatrix}
(p-1)|\tilde R_j|^{p-3}\tilde R_j\Re(\tilde R_j\bar\cdot)+|\tilde R_j|^{p-1}&0\\
0&0
\end{pmatrix}
\right)
\end{gather*}
 and
 $\mathfrak N(\Upsilon)$ is the remaining nonlinear part. To write this equation, we have used Lemma~\ref{lem:interactions-estimates}, the fact that 
\begin{gather*}
E'\left(\sum_{j=1}^N\tilde  R_j\right)=\sum_{j=1}^NE'(\tilde R_j)+O(e^{-3\alpha\sqrt{m-\omega_\star^2}v_\star t}),\\
 \mathfrak L=
J E''\left(\sum_{j=1}^N\tilde R_j\right)+O(e^{-3\alpha\sqrt{m-\omega_\star^2}v_\star t}).
\end{gather*}
and that $\tilde R_j$ is a critical point of $E+\frac{\tilde\omega_j}{\gamma_j}Q+v_j\cdot P$. An analogous computation is derived in all details in Lemma~\ref{lem:taylor}.

Take now the scalar product of~\eqref{eq:linearized-eq} with $iJ\tilde R_{k}$. By using Lemma~\ref{lem:interactions-estimates}, the definition of $\tilde R_k$, and the orthogonality conditions
\eqref{eq:ortho-dyn} it follows that
\begin{align*}
&\psld{\nabla\tilde R_k}{iJ\tilde R_k}=\psld{i \tilde R_k}{iJ\tilde R_k}=0\\
&\psld{\Lambda_{\tilde \omega_j} R_j}{iJ\tilde R_k}=\psld{\nabla\tilde R_j}{iJ\tilde R_k}=\psld{i \tilde R_j}{iJ\tilde R_k}= O(e^{-3\alpha\sqrt{m-\omega_\star^2}v_\star t}) \ \text{ if }j \neq k \\
&\psld{\partial_t \Upsilon}{iJ\tilde R_k}=- \psld{\Upsilon}{\partial_t iJ\tilde R_k}.
\end{align*}
Therefore
\begin{multline}\label{eq:evol1}
\partial_t\tilde\omega_k\left(\Lambda_{\tilde\omega_k} Q(\tilde R_{k})\right)=\psld{\mathfrak L\Upsilon}{iJ\tilde R_k}+\psld{\Upsilon}{\partial_tiJ\tilde R_k}
+O(\norm{\Upsilon}^2_{H^1\times L^2})+O(e^{-3\alpha\sqrt{m-\omega_\star^2}v_\star t}),
\end{multline}
where the term $\Lambda_{\tilde\omega_k} Q(\tilde R_{k})$ comes from 
\[
\psld{\Lambda_{\tilde\omega_k}\tilde R_k}{iJ\tilde R_{k}}=
	\psld{\tilde R_k}{iJ\Lambda_{\tilde\omega_k}\tilde R_k}=
	\Lambda_{\tilde\omega_k} Q(\tilde R_{k}).
\]
Note that by exponential localization
\[
\psld{\mathfrak L\Upsilon}{iJ\tilde R_k}=
\psld{\Upsilon}{\mathfrak L^\star iJ\tilde R_k}=
\psld{\Upsilon}{(JE''(\tilde R_k))^\star iJ\tilde R_k}+O(e^{-3\alpha\sqrt{m-\omega_\star^2}v_\star t})
\]
We want to use the fact that $i\tilde R_k$ belongs to the kernel of $S''(\tilde R_k)$ (see Lemma~\ref{lem:kernel}),
namely that
\[
\left(J\left(E''+\frac{\tilde\omega_k}{\gamma_k}Q''+v_k\cdot P''\right)(\tilde R_k)\right)^\star iJ\tilde R_k=0.
\]
To this aim, we use the definition~\eqref{eq:def-tilde-Phi} of $\tilde R_k$, to compute the following time derivative and make the missing parts appear.
\begin{multline*}
\psld{\Upsilon}{\partial_tiJ\tilde R_k}=\frac{\tilde\omega_k}{\gamma_k}\psld{\Upsilon}{(JQ''(\tilde R_k))^\star iJ\tilde R_k}+v_k\cdot \psld{\Upsilon}{(JP''(\tilde R_k))^\star iJ\tilde R_k}\\
+\left(\frac{\tilde\omega_k}{\gamma_k}-\partial_t\tilde\theta_k\right)\psld{\Upsilon}{J\tilde R_k}
+(v_k-\partial_t\tilde x_k)\cdot \psld{\Upsilon}{ iJ\nabla\tilde R_k}\\
+\partial_t\tilde\omega_k\psld{\Upsilon}{iJ\Lambda_{\tilde\omega_k}\tilde R_k}.
\end{multline*}
Therefore,~\eqref{eq:evol1} gives
\begin{multline}\label{eq:evol2}
\partial_t\tilde\omega_k\left(\Lambda_{\tilde\omega_k} Q(\tilde R_{k})\right)=
	\left(\frac{\tilde\omega_k}{\gamma_k}-\partial_t\tilde\theta_k\right)\psld{\Upsilon}{J\tilde R_k}
+(v_k-\partial_t\tilde x_k)\psld{\Upsilon}{iJ\nabla\tilde R_k}
\\
+\partial_t\tilde\omega_k\psld{\Upsilon}{iJ\Lambda_{\tilde\omega_k}\tilde R_k}
	+O(\norm{\Upsilon}^2_{H^1\times L^2})+O(e^{-3\alpha\sqrt{m-\omega_\star^2}v_\star t}).
\end{multline}
Take the scalar product of~\eqref{eq:linearized-eq} with $\frac{\partial}{\partial x_j}\tilde R_{k}$ for $j=1,\dots, d$ to get from similar arguments
\begin{multline}\label{eq:evol3}
(v_k-\partial_t\tilde x_k)\frac{1}{d}\norm*{\nabla\tilde R_{k}}_2^2=
\psld{\mathfrak L\Upsilon}{\nabla\tilde R_k}+
\psld{\Upsilon}{\partial_t\nabla\tilde R_k}
+O(\norm{\Upsilon}^2_{H^1\times L^2})+O(e^{-3\alpha\sqrt{m-\omega_\star^2}v_\star t}).
\end{multline}
Conversely to what happened for~\eqref{eq:evol2}, we do not expect to have a cancellation on the linear term. 
We just estimate it by
\[
\psld{\mathfrak L\Upsilon}{\nabla\tilde R_k}=
\psld{\Upsilon}{\mathfrak L^\star\nabla\tilde R_k}\leq C\norm{\Upsilon}_2
\]
Therefore~\eqref{eq:evol3} gives 
\begin{equation}\label{eq:evol4}
(v_k-\partial_t\tilde x_k)\frac{1}{d}\norm*{\nabla\tilde R_{k}}_2^2=
O(\norm{\Upsilon}_{H^1\times L^2})+O(e^{-3\alpha\sqrt{m-\omega_\star^2}v_\star t}).
\end{equation}
Finally, take the scalar product of~\eqref{eq:linearized-eq} with $i\tilde R_k$  and argue as previously to obtain 
\begin{equation}\label{eq:evol5}
(\partial_t\tilde\theta_k-\tilde\omega_k)\norm{\tilde R_k}_2^2
=
O(\norm{\Upsilon}_{H^1\times L^2})+O(e^{-3\alpha\sqrt{m-\omega_\star^2}v_\star t}).
\end{equation}
Putting together~\eqref{eq:evol2},~\eqref{eq:evol4} and~\eqref{eq:evol5}  we obtain a differential system for the modulation equations vector $Mod(t):=(\partial_t\tilde\omega_j,\partial_t\tilde\theta_j-\tilde\omega_j,\partial_t\tilde x_j-v_j)_{j=1,\dots,N}$ of the form
\[
A\cdot Mod(t)=M(\Upsilon)++O(e^{-3\alpha\sqrt{m-\omega_\star^2}v_\star t}),
\] 
where $|M(\Upsilon)|\leq C \norm{\Upsilon}_{H^1\times L^2}$.
As long as the modulation parameter do not vary too much and $\norm{\Upsilon}_{H^1}$ remains small, $A$ is invertible (it is of the form $DIAG+small$ with $DIAG$ a diagonal nondegenerate matrix) and we can deduce that 
\begin{equation}\label{eq:first-estimate}
|Mod(t)|\leq C\norm{\Upsilon}_{H^1\times L^2}+O(e^{-3\alpha\sqrt{m-\omega_\star^2}v_\star t}).
\end{equation}
Coming back now to~\eqref{eq:evol2}, it is now easy to see that in fact we can improve in part the previous estimate into
\begin{equation}\label{eq:better-estimate}
\sum_{j=1}^N|\partial_t\tilde\omega_j|\leq C\norm{\Upsilon}_{H^1\times L^2}^2+O(e^{-3\alpha\sqrt{m-\omega_\star^2}v_\star t}).
\end{equation}
This improvement is due to our choice of orthogonality conditions. Combining~\eqref{eq:first-estimate} and~\eqref{eq:better-estimate} gives the desired result.
\end{proof}

\section{Uniform Estimates}\label{sec:uniform}

This Section is devoted to the proof of Proposition~\ref{prop:uniform}. We essentially follow the same line as in the Schr\"odinger case \cite{MaMe06}. Since our approximate multi-solitons have final data $U_n(T_n)=R(T_n)$, they satisfy the desired estimate at least on some interval $[T_n-\delta,T_n]$. Thus the idea is to reduce things to a bootstrap argument: Proposition~\ref{prop:uniform} is a consequence of the following proposition.

\begin{proposition}[Bootstrap]\label{prop:bootstrap}
There exist $\alpha=\alpha(d,N)>0$, and $T_0\in\R$ (independent of $n$) such that for $n$ large enough the following bootstrap property holds. For $t^\dagger\in [T_0,T^n]$, if  $U_n$  satisfies for all $t\in[t^\dagger,T^n]$ the estimate
\begin{equation}\label{eq:bootstrap1}
\norm{U_n(t)-R(t)}_{H^1\times L^2}\leq e^{-\alpha \sqrt{m-\omega_\star^2} v_\star t},
\end{equation}
then it will also satisfies for all $t\in[t^\dagger,T^n]$ the better estimate
\begin{equation}\label{eq:bootstrap2}
\norm{U_n(t)-R(t)}_{H^1\times L^2}\leq \frac{1}{2}e^{-\alpha \sqrt{m-\omega_\star^2} v_\star t}.
\end{equation}
\end{proposition}

Let us quickly indicate how to obtain Proposition~\ref{prop:uniform} from Proposition~\ref{prop:bootstrap}.

\begin{proof}[Proof of Proposition~\ref{prop:uniform}]
Proposition~\ref{prop:bootstrap} implies Proposition~\ref{prop:uniform} by means of  a classical continuity argument (see e.g~\cite{MaMe06}). First, let us notice that the map $t \mapsto U_n(t) \in H^1(\R^d)\times L^2(\R^d)$ is continuous. Second, let us define
$$
t^{\star}:=\inf\{\tau \in [T_0,T^n]\text{ such that }~\eqref{eq:bootstrap1} \text{ holds for all } t \in [\tau,T^n]  \}.
$$
Recall that $U_n(T^n)=R(T^n)$, therefore we have $T_0\leq t^{\star}<T^n$. Our purpose is to show that
$t^{\star}=T_0$. Let us suppose that $t^{\star}>T_0$. Thanks to~\eqref{eq:bootstrap2} we get
$$
\norm{U_n(t)-R(t)}_{H^1\times L^2}\leq \frac{1}{2}e^{-\alpha \sqrt{m-\omega_\star^2} v_\star t},
$$
for all $t\in[t^{\star},T^n]$. By continuity it  exists $\delta_1>0$ such that
$$
\norm{U_n(t)-R(t)}_{H^1\times L^2}\leq e^{-\alpha \sqrt{m-\omega_\star^2} v_\star t},
$$
for all $t\in[t^{\star}-\delta_1,T^n]$. This contradicts the definition of $t^{\star}$ and finishes the proof.
\end{proof}

Hence  now we only have to prove Proposition~\ref{prop:bootstrap}. 
For the rest of the paper, we  make the following assumption.
\begin{BA} Let $T_0>0$ to be determined later and assume that there exists $t^\dagger\in [T_0,T^n]$ such that  $U_n$  satisfies for all $t\in[t^\dagger,T^n]$ the estimate
\begin{equation}\label{eq:bootstrap-assumption}
\norm{U_n(t)-R(t)}_{H^1\times L^2}\leq e^{-\alpha \sqrt{m-\omega_\star^2} v_\star t}.
\end{equation}
\end{BA}

We want to prove that in fact~\eqref{eq:bootstrap-assumption} holds with the better constant $\frac 12$ on the left hand side.

To prove Proposition~\ref{prop:bootstrap}, we need a way to control the difference between the sum of solitons and the approximate multi-soliton $U_n$. If there is only one soliton, it is known since the ground work of Weinstein~\cite{We85} that the coercivity property of the hessian of the action functional (Lemma~\ref{lem:coercivity}) provides a mean to control the difference between a soliton and a solution close to the orbit of the soliton. As in~\cite{CoLe11,CoMaMe11,MaMe06,MaMeTs06}, we are going to generalize such a property to the case of $N$ solitons. To that purpose, we define \emph{localized versions} of the conservation laws around each solitons and prove that a coercivity property also holds for the functional action related to the multi-solitons.

\subsection{Bootstrap}

In this section, we prove Proposition~\ref{prop:bootstrap}, assuming three intermediate Lemmas proved in the later sections.

First of all, we begin by selecting a particular direction of propagation. Define the application $\Omega:\mathbb S^{d-1}\to \R$ by
\[
\Omega(e):=\prod_{j\neq k}|(v_j-v_k,e)_{\R^d}|,\quad e\in\mathbb S^1.
\]
  Let $e_1$ be  such that 
\begin{equation}\label{eq:Omega}
\Omega(e_1)=\max\big\{ \Omega(e),\, e\in\mathbb S^{d-1}
 \big\}>0
\end{equation}
Here, the $\sup$ is a $\max$ since we are maximizing a continuous function on a compact set.  Let us prove the last inequality. We have
\[
\Omega^{-1}(\{0\})=\bigcup_{j\neq k}\left\{ e\in\mathbb S^{d-1},\,(v_j-v_k,e)_{\R^d}=0 \right\}.
\]
Each set composing the union on the right is of $0$ Lesbegue measure, therefore so is $\Omega^{-1}(\{0\})$. Hence $\mathbb S^{d-1}\setminus \Omega^{-1}(\{0\})\neq\emptyset$  and this proves~\eqref{eq:Omega}.
We can complete $e_1$ into an orthonormal basis $(e_1,\dots ,e_d)$ of $\R^d$ and we infer from \eqref{eq:Omega} that there exists $\tilde\alpha>0$ such that for any $j\neq k$, 
\begin{equation}\label{eq:tilde-alpha}
|(v_j-v_k,e_1)_{\R^d}|\geq \tilde \alpha |v_j-v_k|. 
\end{equation}
Since~\eqref{eq:nlkg2} is 
rotation-invariant, we can assume that $(e_1,\dots ,e_d)$ is the canonical basis of $\R^d$. Calling $v_j^1$ the first component of the j-th velocity vector, up to reindexing the solitons, we can assume  that
$$
v_{1}^1<v_{2}^1<\dots <v_{N}^1.
$$
The localization works as follows.
We first define a partition of unity $(\phi_j)_{j=1,\dots,N}$: take $\psi$ a cutoff function such that 
\begin{gather}
\psi(s)=0\text{ for }s<-1\text{ and }\psi(s)=1\text{ for }s>1,\;  0\le \psi'(s)\le 1\text{ in }[-1,1],\nonumber\\
\exists C>0,\,\forall s\in\R,\;|\psi'(s)|\leq C\sqrt{\psi(s)}. \label{eq:cond-IMS}
\end{gather}
Define
\begin{gather*}
\psi_1(t,x)=1,\; \psi_j(t,x)=\psi\left(\frac 1{\sqrt t}(x^1-m_jt)\right), \;m_j=\frac 12 (v_{j-1}^1+v_j^1).
\\
\phi_j=\psi_j-\psi_{j+1}\text{ for }j=1,\dots,N-1\text{ and }\phi_N=\psi_N.
\end{gather*}
We consider the following \emph{localized action} functional for $W=(w_1,w_2)\in H^1(\R^d)\times L^2(\R^d)$
\[
{\mathcal{S}}(t,W)=\sum_{j=1}^N S_j(t,W)=\sum_{j=1}^N E_j(t,W)+\frac {\omega_j}{\gamma_j}Q_j(t,W)+v_j\cdot P_j(t,W),
\]
where for $j=1,\dots,N$ we have defined the \emph{localized energies, charges and momenta} by
\begin{align*}
 E_j(t,W)&:=\frac 12\int_{\R^d} |\nabla w_1|^2\phi_jdx+
\frac m2\int_{\R^d} | w_1|^2\phi_jdx+\frac 12\int_{\R^d} | w_2|^2\phi_jdx 
-\frac 1{p+1}\int_{\R^d} |w_1|^{p+1}\phi_jdx,\\
Q_j(t,W) &:=\Im  \int_{\R^d} w_1 \bar w_2 \phi_j dx,\\ 
P_j(t,W)&:= \Re  \int_{\R^d} \nabla w_1 \bar w_2 \phi_j dx .
\end{align*}

Since $U_n$ verifies~\eqref{eq:bootstrap-assumption}, we can assume that $T_0$ is large enough, so that $U_n$ satisfies the hypotheses of Proposition~\ref{prop:dyn-mod} and thus there exists a modulated sum of solitons $\tilde R=\sum_{j=1}^N\tilde R_j$ and $\Upsilon_n$ verifying the orthogonality conditions~\eqref{eq:ortho-dyn} such that 
\begin{align}\label{eq:Y-error}
&U_n(t)=\tilde R(t)+\Upsilon_n(t), \nonumber\\
&\norm{\Upsilon_n}_{H^1 \times L^2}\leq Ce^{-\alpha \sqrt{m-\omega_\star^2} v_\star t}.
\end{align}
Let us define the localized linearized action for $\tilde R$ by
\[
\mathcal H_n(\Upsilon_n(t),\Upsilon_n(t)):=\sum_{j=1}^N\dual{S_j''(\tilde R_j(t))\Upsilon_n(t)}{\Upsilon_n(t)}.
\]

It turns out that $\mathcal H_n$ is inheriting the coercivity property of the hessian of the action around a single soliton (Lemma~\ref{lem:coercivity}).

\begin{lemma}[Coercivity]\label{lem:multi-coercivity}
There exists $C>0$ such that  for all $t\in[T_0,T^n]$ the localized Hessian verifies
\[
\mathcal H_n(\Upsilon_n(t),\Upsilon_n(t))\geq C\norm{\Upsilon_n(t)}^2_{H^1\times L^2}.
\]
\end{lemma}

In addition, since 
$\mathcal S(t,U_n(t))$ is made of localized versions of conserved quantities, it varies slowly.

\begin{lemma}[Almost conservation]\label{lem:almost-conservation}
For $t\in[t^{\star},T^n]$, we have 
\[
\left|\frac {\partial}{\partial t}\mathcal S(t,U_n(t))\right|\leq 
o(e^{-2\alpha \sqrt{m-\omega^2_\star}v_\star t}).
\]
\end{lemma}
We also have the following Taylor-like expansion for  $\mathcal S(t,U_n(t))$.

\begin{lemma}[Taylor-like expansion]\label{lem:taylor}
The action $\mathcal S(t,U_n(t)) $ satisfies for $t\in[t^{\star},T^n]$
\begin{multline*}
\mathcal S(t,U_n(t)) =\sum_{j=1}^N \left(E(R_j(t))+\frac{\omega_j}{\gamma_j} Q(R_j(t))+v_j\cdot P(R_j(t))\right)\\
+\mathcal H_n(\Upsilon_n(t),\Upsilon_n(t))
+o(e^{-2\alpha \sqrt{m-\omega_\star^2}v_\star t}),
\end{multline*}
\end{lemma}

With Lemmas~\ref{lem:multi-coercivity},~\ref{lem:almost-conservation} and~\ref{lem:taylor} in hand, we can now conclude the proof of Proposition~\ref{prop:bootstrap}.

\begin{proof}[Proof of Proposition~\ref{prop:bootstrap}]
The first step is to show that
\begin{equation}\label{eq:esti-Upsilon}
\norm{\Upsilon_n}_{H^1 \times L^2}^2=o\left(e^{-2\alpha\sqrt{m-\omega_{\star}^2}v_{\star}t}\right).
\end{equation}
Indeed, thanks to Lemma~\ref{lem:almost-conservation} we obtain 
\begin{equation}\label{eq:bored}
 \mathcal S(t,U_n(t))\leq \mathcal S(T^n,U_n(T^n))
+o(e^{-2\alpha\sqrt{m-\omega_{\star}^2}v_{\star}t}).
\end{equation}
Now notice that $\sum_{j=1}^N\left(E(R_j(t))+\frac{\omega_j}{\gamma_j} Q(R_j(t))+v_j\cdot P(R_j(t))\right)$ is a time independent quantity. Therefore, 
\begin{multline*}
\sum_{j=1}^N \left(E(R_j(t))+\frac{\omega_j}{\gamma_j} Q(R_j(t))+v_j\cdot P(R_j(t))\right)
\\
=\sum_{j=1}^N \left(E(R_j(T_n))+\frac{\omega_j}{\gamma_j} Q(R_j(T_n))+v_j\cdot P(R_j(T_n))\right)
=\mathcal S(T_n,U_n(T_n)).
\end{multline*}
Combined with Lemma~\ref{lem:taylor} and~\eqref{eq:bored}, this implies
\begin{multline*}
\mathcal H(\Upsilon_n(t),\Upsilon_n(t))
=\mathcal S(t,U_n(t))-\mathcal S(T_n,U_n(T_n))+o\left(e^{-2\alpha\sqrt{m-\omega_{\star}^2}v_{\star}t}\right)
=
o\left(e^{-2\alpha\sqrt{m-\omega_{\star}^2}v_{\star}t}\right).
\end{multline*}
By Lemma~\ref{lem:multi-coercivity} we get
 \[
C\norm{\Upsilon_n}_{H^1 \times L^2}^2\leq \mathcal H(\Upsilon_n(t),\Upsilon_n(t))
\leq
o\left(e^{-2\alpha\sqrt{m-\omega_{\star}^2}v_{\star}t}\right).
\]
Hence~\eqref{eq:esti-Upsilon} is proved.
Now we have
$$
\norm{U_n-R}^2_{H^1\times L^2}\leq 2\norm{\tilde R-R}^2_{H^1\times L^2}+2\norm{\Upsilon_n}^2_{H^1\times L^2},
$$
such that, by~\eqref{eq:esti-Upsilon} and~\eqref{eq:parameters-derivatives-estimate} we infer
\begin{multline*}
\norm{U_n-R}^2_{H^1\times L^2}
\leq C\bigg(\sum_{j=1}^N|\tilde \omega_j(t)-\omega_j|^2+|\tilde \theta_j(t)-\theta_j|^2+|\tilde x_j(t)-x_j|^2\bigg)
+o\left(e^{-2\alpha\sqrt{m-\omega_{\star}^2}v_{\star}t}\right)
\\
\leq 
o\left(e^{-2\alpha\sqrt{m-\omega_{\star}^2}v_{\star}t}\right).
\end{multline*}
Choosing $t$ large enough we have
$$
\norm{U_n-R}_{H^1\times L^2}\leq\frac{1}{2}e^{-\alpha\sqrt{m-\omega_{\star}^2}v_{\star}t}.
$$
This concludes the proof.
\end{proof}

\subsection{Coercivity}

\emph{From now on and until the end of this paper, the subscript $n$ is removed  when there is no possible confusion. For example,  $U_n$ is now denoted simply by $U$.
}

We first prove  Lemma~\ref{lem:multi-coercivity}.

\begin{proof}[Proof of Lemma~\ref{lem:multi-coercivity}]
From Lemma~\ref{lem:coercivity}, we already know that for any $j=1,\dots,N$ we have
\[
\dual{S''(\tilde R_j(t))\Upsilon(t)}{\Upsilon(t)}\geq C\norm{\Upsilon(t)}^2_{H^1\times L^2}
\]
where the dependency of $S$ in $j$ is understood (recall that $S=E+\frac{\tilde \omega_j}{\gamma_j}Q+v_j\cdot P$).
We remark that 
\begin{multline*}
\dual{Q''_j(\tilde R_j)\Upsilon(t)}{\Upsilon(t)}=\Im\int_{\R^d}\Upsilon_1(t)\bar\Upsilon_2(t)\phi_jdx\\
=\Im\int_{\R^d}(\sqrt{\phi_j}\Upsilon_1(t)\overline{(\sqrt{\phi_j}\Upsilon_2(t))}dx=\dual{Q''(\tilde R_j)\sqrt{\phi_j}\Upsilon(t)}{\sqrt{\phi_j}\Upsilon(t)}.
\end{multline*}
Similar computations can be performed for the momentum and the $0$-order part of the energy. We deal with the gradient part by means of the classical IMS localization formula (see e.g. \cite{Si82}):
\[
\norm{\nabla \Upsilon_1}_2^2= \sum_{j=1}^N\left(\norm*{\nabla\left(\sqrt{\phi_j}\Upsilon_1\right)}_2^2-\norm*{\left|\nabla\left(\sqrt{\phi_j}\right)\right|\Upsilon_1}_2^2\right).
\]
Straightforward computations using the definition of the cutoff functions $\phi_j$  and~\eqref{eq:cond-IMS} imply that 
 \[
\norm*{\nabla\left(\sqrt{\phi_j}\right)}_\infty\leq \frac{C'}{\sqrt{t}}
\]
This implies that 
\[
\norm*{\nabla\left(\sqrt{\phi_j}\right)\Upsilon_1}_2^2\leq \frac{C'}{\sqrt{t}}\norm{\Upsilon}^2_{H^1\times L^2}.
\]
Combining these informations,
 we infer that
\begin{multline*}
\mathcal H(\Upsilon(t),\Upsilon(t))=
\sum_{j=1}^N\dual{S_j''(\tilde R_j(t))\Upsilon(t)}{\Upsilon(t)}=\\
\sum_{j=1}^N\dual{S''(\tilde R_j(t))\sqrt{\phi_j}\Upsilon(t)}{\sqrt{\phi_j}\Upsilon(t)}-\norm*{\nabla\left(\sqrt{\phi_j}\right)\Upsilon_1}_2^2\\
\geq \left(C-\frac{C'}{\sqrt{t}}\right)\norm{\Upsilon(t)}^2_{H^1\times L^2} \geq \frac C2\norm{\Upsilon(t)}^2_{H^1\times L^2},
\end{multline*}
where the last inequality follows from the fact that $t\geq T_0$ and $T_0$ can be chosen so that $T_0\geq \left(\frac{2C'}{C}\right)^2$.
\end{proof}

\subsection{Almost conservation}

In this section, we prove Lemma~\ref{lem:almost-conservation}. Recall that we have assumed that $U\equiv U_n$ verifies the bootstrap assumption~\eqref{eq:bootstrap-assumption}.
We start with a preliminary lemma.

\begin{lemma}\label{lem:deriv-local}
Let $\phi$ be a $\mathcal C^1$ function of the variable $x_1$ such that $\phi$ and $\phi'$ are bounded . Then for all $t$ in the time interval of existence of $U$ we have
\begin{align*}
\frac{\partial}{\partial t}\Im\int_{\R^d}u_1\bar u_2\phi(x_1)dx&=\Im\int_{\R^d}\partial_{x_1}u_1\bar u_1\phi'(x_1)dx,\\
\frac{\partial}{\partial t}\Re\int_{\R^d}\partial_{x_j}u_1\bar u_2\phi(x_1)dx&=-\Re\int_{\R^d}\partial_{x_j}u_1\partial_{x_1}\bar u_1\phi'(x_1)dx \quad(j\neq1),\\
\frac{\partial}{\partial t}\Re\int_{\R^d}\partial_{x_1}u_1\bar u_2\phi(x_1)dx&=\int_{\R^d}\bigg(-|\partial_{x_1}u_1|^2+\frac{1}{2}\left(|\nabla u_1|^2+m|u_1|^2-|u_2|^2\right)
\nonumber\\
	&\hspace{3cm}
-\frac{1}{p+1}|u_1|^{p+1}\bigg)\phi'(x_1)dx.
\end{align*}
\end{lemma}

\begin{proof}
The results follows from elementary computations using the fact that $U$ is a solution to~\eqref{eq:nlkg2}. 
\end{proof}

\begin{proof}[Proof of Lemma~\ref{lem:almost-conservation}]
Let $U=(u_1,u_2)$ and let us start by looking at the derivative of the localized charge. By Lemma~\ref{lem:deriv-local} we have 
\begin{eqnarray*}
  \left|\frac{\partial}{\partial t}\Im  \int_{\R^d} u_1 \bar u_2 \psi_j dx\right|&=&\left|\Im  \int_{\R^d} \partial_{x_1} u_1 \bar u_1   \partial_{x_1}\psi_j dx + 
\Im  \int u_1 \bar u_2  \partial_{t}\psi_j dx\right|\\
&=& \frac1{\sqrt t}\left|\Im\int_{\R^d} \left(\partial_{x_1} u_1 \bar u_1 
-\frac{m_j}{2}u_1 \bar u_2 \right)\psi'\left(\frac {x^1-m_jt}{\sqrt{t}}\right)dx\right|
\\
&\leq&\frac{C}{\sqrt{t}}\int_{\tilde A_j} (|\partial_{x_1} u_1 \bar u_1|+|u_1 \bar u_2|)dx,
\end{eqnarray*}
where $\tilde A_j:=\{ x \in \R^d;\;\psi_j'(x)\neq 0 \}$. Remembering that
$\phi_j=\psi_j-\psi_{j+1}$ for $j=1,\dots,N-1$ and $\phi_N=\psi_N$,
and defining $\ A_j:=\{ x \in \R^d;\;\ \phi_j' \neq 0 \}$ we have
$$ 
\left|\frac{\partial}{\partial t}\Im  \int_{\R^d} u_1 \bar{u_2} \phi_j dx\right|\leq \frac{C}{\sqrt{t}}\int_{A_j} (|\partial_{x_1} u_1 \bar u_1|+|u_1 \bar u_2|)dx.
$$
and then
\begin{equation*}
\left|\frac{\partial}{\partial t}\Im  \int u_1 \bar{u_2} \phi_j dx\right|\leq \frac{C}{\sqrt{t}}\left(\|u_1\|_{H^1(A_j)}^2+\|u_2\|_{L^2(A_j)}^2\right).
\end{equation*}
Now notice that $\norm{U}^2_{H^1(A_j)\times L^2(A_j)}\leq 2\norm{U-R}^2_{H^1(\R^d)\times L^2(\R^d)}+2\norm{R}^2_{H^1(A_j)\times L^2(A_j)}$. 
Thanks to Lemma~\ref{lem:interactions-estimates} we have
$$
\norm{R}^2_{H^1(A_j)\times L^2(A_j)}\leq \norm{R_j}^2_{H^1(A_j)\times L^2(A_j)}+O(e^{-3\alpha\sqrt{m-\omega_\star^2}v_\star t}).
$$
By using the properties of our partition of unity and the decay of the profile of $R_j$
it follows that
$$
\norm{R_j}^2_{H^1(A_j)\times L^2(A_j)}\leq \int_{|x^1|\geq\frac{\tilde\alpha}{2}v_{\star}t}Ce^{-\frac 12\sqrt{m-\omega_{\star}^2}|x|}dx\leq Ce^{-\frac {\tilde\alpha}{4} \sqrt{m-\omega^2_\star}v_\star t}.
$$
Recall that $\tilde\alpha$ stems from~\eqref{eq:tilde-alpha}. 
We conclude thanks to the bootstrap assumption~\eqref{eq:bootstrap-assumption} that 
\begin{equation}\label{eq:der-q}
\left|\frac {\partial}{\partial t} Q_j(t,U(t))\right|= o(e^{-2\alpha \sqrt{m-\omega^2_\star}v_\star t})
\end{equation}
if we choose $\alpha<\frac{\tilde\alpha}{8}.$
Now we focus on the derivative of the localized momenta. We start with the first component of the momentum. By Lemma~\ref{lem:deriv-local} we have 
\begin{multline*}
\left|\frac{\partial}{\partial t} \Re  \int \partial_{x_1}u_1 \bar u_2\psi_jdx\right|\\
\le \frac {1}{\sqrt t}\int_{\R^d}  \bigg[ \partial_{x_1}u_1\bar u_2-
\left( -\frac{m_1}{2}|\partial_{x_1}  u_1|^2+\frac {|u_2|^2}2-\frac {|\nabla u_1|^2}2-\frac {m|u_1|^2}2\right)\\
\shoveright{- \frac{|u_1|^{p+1}}{p+1}\bigg]\psi'\left(\frac {x_1-m_1t}{\sqrt{t}}\right)dx}
 \\
\le \frac {C}{\sqrt t}\left[ \|u_1\|_{H^1(A_j)}^2+\|u_2\|_{L^2(A_j)}^2+\|u_1\|^{p+1}_{L^{p+1}(A_j)}\right] \\
\le \frac {C}{\sqrt t}\left[ \|u_1\|_{H^1(A_j)}^2+\|u_2\|_{L^2(A_j)}^2+\|u_1\|^{p+1}_{H^1(A_j)}\right].
\end{multline*}
Now we argue as for the derivative of the localized charge. The other components of the momentum can be estimated in a similar fashion. Thus we have
\begin{equation}\label{eq:der-p}
\left| \frac{\partial}{\partial t} P_j(t,U(t))\right|= o(e^{-2\alpha \sqrt{m-\omega^2_\star}v_\star t}).
\end{equation}
Remark now that 
\[
\mathcal S(t,U(t))=E(U(t))+\sum_{j=1}^N \frac{\omega_j}{\gamma}Q_j(t,U(t))+v_j\cdot P_j(t,U(t)).
\]
Since $E$ is a conserved quantity, combining~\eqref{eq:der-q} and~\eqref{eq:der-p} gives the desired result. 
\end{proof}

\subsection{The Taylor expansion}

We now prove Lemma~\ref{lem:taylor}. We start by an estimate on the modulation parameters.

\begin{lemma} \label{lem:claim3} For any $t\in[t^{\star},T^n]$, we have
\[
\sum_{j=1}^N|\tilde \omega_j(t)-\omega_j|=O(e^{-2\alpha\sqrt{m-\omega_\star^2}v_\star t}).
\]
\end{lemma}

\begin{proof}
Recall that $U=\sum_{j=1}^N\tilde R_j+\Upsilon$. Thanks to the interaction estimates given by Lemma~\ref{lem:interactions-estimates} and the orthogonality conditions it follows 
\begin{equation}\label{eq:e1}
Q_j(t,U)=Q(\tilde R_j)+\Im \int_{\R^d}\Upsilon_1\bar\Upsilon_2\phi_jdx+O(e^{-3\alpha \sqrt{m-\omega_{\star}^2}v_{\star}t}).
\end{equation}
We already computed the time-derivative of $Q_j$ during the proof of Lemma~\ref{lem:almost-conservation} (see~\eqref{eq:der-q}), and it implies 
\begin{equation}\label{eq:e2}
|Q_j(t,U(t))-Q_j(T^n,U(T_n))|=o(e^{-2\alpha \sqrt{m-\omega^2_\star}v_\star t}).
\end{equation}
Thanks to the scaling property \eqref{eq:scaled-phi} of the profile  we get
\begin{multline*}
Q(\tilde R_j(t))-Q(\tilde R_j(T_n))\\
=\gamma_j\left(-\tilde \omega_j(t)
\left(m-\tilde \omega_j(t)^2\right)^{\frac{2}{p-1}-\frac{d}{2}}
+\tilde \omega_j(T^n)\left(m-\tilde \omega_j(T^n)^2\right)^{\frac{2}{p-1}-\frac{d}{2}}\right)\norm{\tilde \varphi(x)}_2^2,\\
=\gamma_j\left(-\tilde \omega_j(t)
\left(m-\tilde \omega_j(t)^2\right)^{\frac{2}{p-1}-\frac{d}{2}}
+\tilde \omega_j\left(m-\tilde \omega_j^2\right)^{\frac{2}{p-1}-\frac{d}{2}}\right)\norm{\tilde \varphi(x)}_2^2,
\end{multline*}
where the last inequality is due to the fact that $\tilde \omega_j(T^n)=\omega_j$.
By simple Taylor expansion in frequencies we conclude
\begin{multline}\label{eq:e3}
Q(\tilde R_j(t))-Q(\tilde R_j(T_n))\\
\gamma_j\left[-(m-\omega_j^2)^{\frac{2}{p-1}-\frac{d}{2}}+2\omega_j^2(\frac{2}{p-1}-\frac{d}{2})(m-\omega_j^2)^{\frac{2}{p-1}-\frac{d}{2}-1}\right](\tilde \omega_j(t)-\omega_j)\norm{\tilde \varphi(x)}_2^2\\
+o(\tilde \omega_j(t)-\omega_j).
\end{multline}
Since $\omega_j$ is part of the set $\mathcal O_{\rm stab}$ (see \eqref{eq:O-stab}) we have
$$
-(m-\omega_j^2)^{\frac{2}{p-1}-\frac{d}{2}}+2\omega_j^2(\frac{2}{p-1}-\frac{d}{2})(m-\omega_j^2)^{\frac{2}{p-1}-\frac{d}{2}-1}>0.
$$
Combining the bootstrap assumption \eqref{eq:bootstrap-assumption}, and \eqref{eq:e1}-\eqref{eq:e3} gives the desired result.
\end{proof}

\begin{proof}[Proof of Lemma~\ref{lem:taylor}]
The first step consists in splitting the action using $U=\sum_{j=1}^N\tilde R_j+\Upsilon$.
We start with the energy part. We have
\begin{multline*}
\sum_{j=1}^N E_j(t,U)=E(U)=E(\tilde R+\Upsilon)
=E(\tilde R)+E'(\tilde R)\Upsilon+\frac12\dual{E''(\tilde R)\Upsilon}{\Upsilon}+o(\norm{\Upsilon}_{H^1\times L^2}^2).
\end{multline*}
We treat the $0$ order term first. By Lemma~\ref{lem:interactions-estimates}, we have
\begin{multline*}
E(\tilde R)
=\frac12\norm[\Big]{\nabla \Big(\sum_{j=1}^N \tilde R_{j,1} \Big)}_2^2+\frac m2\norm[\Big]{ \sum_{j=1}^N \tilde R_{j,1} }_2^2+\frac12\norm[\Big]{\sum_{j=1}^N \tilde R_{j,2} }_2^2-\frac{1}{p+1}\norm[\Big]{ \sum_{j=1}^N \tilde R_{j,1} }_{p+1}^{p+1}\\
=\sum_{j=1}^N\left(\frac12\norm{\nabla  \tilde R_{j,1} }_2^2+\frac m2\norm{ \tilde R_{j,1} }_2^2+\frac12\norm{\tilde R_{j,2} }_2^2-\frac{1}{p+1}\norm{  \tilde R_{j,1} }_{p+1}^{p+1}\right)
+O(e^{-3\alpha \sqrt{m-\tilde \omega_\star^2}v_\star t})
\end{multline*}
where $\tilde \omega_\star=\max\{|\tilde \omega_j|; j=1,\dots,N\}$. In short, we have
\[
E(\tilde R)=\sum_{j=1}^N E(\tilde R_j)+O(e^{-3\alpha \sqrt{m-\tilde \omega_\star^2}v_\star t}).
\]
Now notice that 
$$
e^{-3\alpha \sqrt{m- \omega_{\star}^2}v_\star t}-e^{-3\alpha \sqrt{m-\tilde \omega_\star^2}v_\star t}=\frac{-\omega_\star e^{-3\alpha \sqrt{m-\omega_{\star}^2}v_\star t}(3\alpha v_{\star}t)}{\sqrt{m-\omega_{\star}^2}}(\tilde \omega_{\star}-\omega_{\star})+o(|\tilde \omega_{\star}-\omega_{\star}|),
$$
such that, thanks to Lemma~\ref{lem:claim3}, we get 
$$
E(\tilde R)=\sum_{j=1}^N E(\tilde R_j)+O(e^{-3\alpha \sqrt{m-\omega_\star^2}v_\star t}).
$$
Using similar arguments we have
\begin{align*}
E'(\tilde R)\Upsilon&=\sum_{j=1}^N E'(\tilde R_j)\Upsilon+O(e^{-3\alpha \sqrt{m-\omega_\star^2}v_\star t}),\\
\dual{E''(\tilde R)\Upsilon}{\Upsilon}&=\sum_{j=1}^N \dual{E''(\tilde R_j)\Upsilon}{\Upsilon}+O(e^{-3\alpha \sqrt{m-\omega_\star^2}v_\star t}).
\end{align*}
The proof follows the same steps for the localized charges and momenta: we have
\begin{align*}
Q_j(U)&=\sum_{j=1}^N \left( Q(\tilde R_j)+Q'(\tilde R_j)\Upsilon+\frac12\dual{Q_j''(\tilde R_j)\Upsilon}{\Upsilon}\right)+O(e^{-3\alpha \sqrt{m-\omega_\star^2}v_\star t}),
\\
P_j(U)&=\sum_{j=1}^N \left( P(\tilde R_j)+P'(\tilde R_j)\Upsilon+\frac12\dual{P_j''(\tilde R_j)\Upsilon}{\Upsilon}\right)+O(e^{-3\alpha \sqrt{m-\omega_\star^2}v_\star t}).
\end{align*}
The second step consists in expanding $\tilde\omega_j$ around $\omega_j$ using Lemma~\ref{lem:claim3}.
Remembering that $\tilde R_j$ is a critical point of $E+\frac{\tilde \omega_j}{\gamma_j}Q+v_j\cdot P$, we infer
\[
E'(\tilde R_j)+\frac{\omega_j}{\gamma_j}Q'(\tilde R_j)+v_j\cdot P'(\tilde R_j)=\frac{\omega_j-\tilde\omega_j}{\gamma_j}Q'(\tilde R_j).
\]
From Lemma~\ref{lem:claim3},~\eqref{eq:bootstrap-assumption} and~\eqref{eq:Y-error}, it follows that
\[
\left|\frac{\omega_j-\tilde\omega_j}{\gamma_j}Q'(\tilde R_j)\Upsilon\right|\leq 
O(e^{-2\alpha\sqrt{m-\omega_\star^2}v_\star t})\norm{\Upsilon}_{H^1\times L^2}\leq O(e^{-3\alpha\sqrt{m-\omega_\star^2}v_\star t}).
\]
The only thing left to see is to remove the tildes corresponding to modulation. We have
\begin{multline*}
\sum_{j=1}^N\left(E(\tilde R_j)+\frac{\omega_j}{\gamma_j}Q(\tilde R_j)+v_j\cdot P_j(\tilde R_j)\right)=\\
\sum_{j=1}^N\bigg(E( R_j)+\frac{\omega_j}{\gamma_j}Q( R_j)+v_j\cdot P_j( R_j)+
O((\tilde\omega_j-\omega_j)^2)\bigg),
\end{multline*}
where we have used the fact that
\[
(\tilde\omega_j-\omega_j)(E'+\frac{\omega_j}{\gamma_j}Q'+v_j\cdot P')(R_j)\frac{\partial R_j}{\partial \omega}=0.
\]
Thanks to Lemma~\ref{lem:claim3} we have
$$
\sum_{j=1}^N |\tilde\omega_j-\omega_j|^2\leq O(e^{-4\alpha \sqrt{m- \omega_\star^2}v_\star t}).
$$
Gathering all these informations we get the desired result.
\end{proof}

\begin{ack}
J.B is supported  by  FIRB2012 `Dinamiche dispersive: analisi di Fourier e metodi variazionali' and  PRIN2009 `Metodi Variazionali e Topologici nello Studio di Fenomeni non Lineari', M.G  by the PRIN2009 grant `Critical Point Theory
and Perturbative Methods for Nonlinear Differential Equations', S.L.C by the french ANR project ESONSE. 
\end{ack}

\appendix

\section{Appendix}

\begin{lemma}[Rellich-Kondrachov in $H^s$]\label{lem:rellich}
Let $\Omega$ be a bounded open set, $s\geq 0$ and $\varepsilon>0$, and $u_n\in H^s(\R^d)$ be a bounded sequence such that $\text{ supp } u_n\subset \Omega$. Then there exists $u\in H^s$ such that  $\norm{u_n-u}_{H^{s-\varepsilon}}=o(1)$
\end{lemma}

\begin{proof}
Let $u_n$ be a bounded sequence in $H^s(\Omega)$ weakly converging to $u\in H^s$,
we shall prove that, up to subsequences, $\norm{u_n-u}_{H^{s-\varepsilon}}=o(1)$.
By Plancherel identity we have
\begin{multline*}
\norm{u_n-u}_{H^{s-\varepsilon}}^2=\int_{|\xi|\leq R} (1+|\xi|^2)^{s-\varepsilon}|\hat u_n(\xi)-\hat u(\xi)|^2d \xi
+\int_{|\xi|> R}(1+|\xi|^2)^{s-\varepsilon}|\hat u_n(\xi)-\hat u(\xi)|^2d \xi.
\end{multline*}
We have
\begin{multline*}
\int_{|\xi|> R} (1+|\xi|^2)^{s-\varepsilon}|\hat u_n(\xi)-\hat u(\xi)|^2d \xi \leq \\
\frac{1}{(1+R^2)^{\varepsilon}}\int (1+|\xi|^2)^{s}|\hat u_n(\xi)-\hat u(\xi)|^2d \xi\leq  
\frac{2}{(1+R^2)^{\varepsilon}}\norm{u_n}_{H^s}^2.
\end{multline*}
in addition $\Omega$ is bounded and by weak convergence we have $\hat u_n(\xi) \rightarrow \hat u (\xi)$. To conclude it suffices to show that
\begin{equation}\label{eq:ts}
\int_{|\xi|\leq  R} (1+|\xi|^2)^{s-\varepsilon}|\hat u_n(\xi)-\hat u(\xi)|^2d \xi =o(1).
\end{equation}
Notice that 
\[
\norm{\hat u_n}_{L^\infty(\Omega)}\leq \norm{u_n}_{L^1(\Omega)}\leq 
\mu(\Omega)^{\frac 12} \norm{u_n}_{L^2(\Omega)}\leq \mu(\Omega)^{\frac 12}\norm{u_n}_{H^s}
\]
and hence  $(1+|\xi|^2)^{s-\varepsilon}|\hat u_n(\xi)-\hat u(\xi)|^2$ is dominated by $C(1+|R|^2)^{s-\varepsilon}$ such that
\eqref{eq:ts} holds. \\
Now, fix $\delta>0$ and choose $R>0$ and $N$ sufficiently large such that 
\[
\int_{|\xi|> R} (1+|\xi|^2)^{s-\varepsilon}|\hat u_n(\xi)-\hat u(\xi)|^2d \xi\leq \frac{\delta}{2},
\]
and for all $n\geq N$
\[
\int_{|\xi|\leq  R} (1+|\xi|^2)^{s-\varepsilon}|\hat u_n(\xi)-\hat u(\xi)|^2d \xi \leq \frac{\delta}{2},
\]
i.e $\norm{u_n-u}_{H^{s-\varepsilon}}\leq \delta.$
\end{proof}

\begin{lemma}[Interactions estimates]\label{lem:interactions-estimates}
There exists $f\in L^\infty_tL^1_x(\R,\R^d)\cap L^\infty_tL^\infty_x(\R,\R^d)$ such that if $j\neq k$
\begin{gather*}
|R_jR_k|+|R_j\nabla R_k|+|\nabla R_j\nabla R_k|
+|R_j|\phi_k+|\nabla R_j|\phi_k
\leq Ce^{-3\alpha\sqrt{m-\omega_\star^2}v_\star t}f(t,x)
\\
\begin{split}
\left||R|^{p+1}-\sum_{l=1}^N|R_l|^{p+1}\right|
+\left||R|^{p-1}R-\sum_{l=1}^N|R_l|^{p-1}R_l\right|
+\left||R|^{p-1}-\sum_{l=1}^N|R_l|^{p-1}\right|
\\
\leq Ce^{-3\alpha\sqrt{m-\omega_\star^2}v_\star t}f(t,x).
\end{split}
\end{gather*}
\end{lemma}

\begin{proof}
We start proving that there exists  $f\in L^\infty_tL^1_x(\R,\R^d)\cap L^\infty_tL^\infty_x(\R,\R^d)$ such that if $j\neq k$ 
$$|R_jR_k|\leq  Ce^{-3\alpha\sqrt{m-\omega_\star^2}v_\star t}f(t,x).$$
Thanks to~\eqref{eq:exponential-decay}
(the Lorenz transform gives indeed only a contraction along the direction of propagation) we know that
\begin{gather*}
|R_j|\leq Ce^{-\frac 12\sqrt{m-\omega_j^2}|x-v_jt|}\leq Ce^{-\frac 12 \sqrt{m-\omega_{\star}^2}|x-v_jt|}\\
|R_k|\leq Ce^{-\frac 14 \sqrt{m-\omega_k^2}|x-v_jt|}\leq Ce^{-\frac 14 \sqrt{m-\omega_{\star}^2}|x-v_kt|}.
\end{gather*}
By a simple change of variable we get
$$|R_j||R_k|\leq Ce^{-\frac 12\sqrt{m-\omega_\star^2}|x|}e^{-\frac 14\sqrt{m-\omega_\star^2}|x-(v_k-v_i)t)|},$$
such that, thanks to the following inequality 
$$|x-(v_k-v_j)t|\geq |(v_k-v_j)t|-|x|\geq v_\star t-|x|$$
we conclude
$$|R_j||R_k|\leq Ce^{-\frac 14\sqrt{m-\omega_\star^2}|x|}e^{-\frac 14\sqrt{m-\omega_\star^2}v_\star t}.$$
Taking $3\alpha\leq \frac{1}{4}$ we get the desired estimate. The estimates for
$|R_j\nabla R_k|$ and $|\nabla R_j\nabla R_k|$ follow analogously.\\
Now we shall prove that if $j\neq k$
$$|R_j|\phi_k\leq  Ce^{-3\alpha\sqrt{m-\omega_\star^2}v_\star t}f(t,x).$$
Let us suppose without any lack of generality that $j<k-1$. Notice that
$$|R_j|\phi_k\leq C^{-\frac 12 \sqrt{m-\omega_\star^2}|x-v_jt|}\chi_{[\frac{1}{2}(v_{k-1}+v_{k})t-\sqrt{t}, \frac{1}{2}(v_{k+1}+v_{k})t+\sqrt{t}]}$$ 
that implies
$$|R_j|\phi_k\leq Ce^{-\frac 12 \sqrt{m-\omega_\star^2}|x-v_jt|}\chi_{[(v_{j}+v_{\star})t-\sqrt{t}, v_k t+\sqrt{t}]}.$$
By a simple change of variable we get
$$|R_j|\phi_k\leq Ce^{-\frac 12 \sqrt{m-\omega_\star^2}|x|}\chi_{[v_{\star}t-\sqrt{t}, (v_{k}-v_j)t+\sqrt{t}]}.$$
Now, for $t\geq \max\{\frac{4}{v_\star^2}, 1\}$, it follows
\begin{multline}
|R_j|\phi_k\leq Ce^{-\frac 14 \sqrt{m-\omega_{\star}^2}|x|}e^{-\frac 14 \sqrt{m-\omega_{\star}^2}|x|}\chi_{[\frac 12 v_{\star}t, (v_{k}-v_j+1)t]} 
\leq Ce^{-\frac 18 \sqrt{m-\omega_{\star}^2}v_\star t}e^{-\frac 14 \sqrt{m-\omega_{\star}^2}|x|}.
\end{multline}
Now for $\alpha<\frac{1}{24}$ we conclude
$$|R_j|\phi_k\leq Ce^{- 3\alpha \sqrt{m-\omega_{\star}^2}v_\star t}e^{-\frac 14 \sqrt{m-\omega_{\star}^2}|x|}.$$
The case $j\geq k-1$, $j\neq k$, follows identically as well as the estimates
concerning the gradient.
The second part of the lemma follows from the inequality
\begin{equation*}
(|a+b|^p-|a|^p-|b|^p)\leq C (|a||b|^{p-1}+|a|^{p-1}|b|) \ \text{ with } p>0
\end{equation*}
that derives from the elementary inequality
\begin{equation*}
(|1+t|^p-1-|t|^p)\leq C (|t|+|t|^{p-1}) \ \text{ with } p>0.
\end{equation*}
By arguing as before we get the desired estimates.
\end{proof}

\bibliographystyle{siam}
\bibliography{biblio}

%

\end{document}